\documentclass[12pt,a4paper,psamsfonts,reqno]{amsart}

\usepackage{fancyhdr}
\usepackage{appendix}
\usepackage{amssymb,amscd,amsxtra,calc}
\usepackage{graphicx} % Required for inserting images
\usepackage{tikz-cd}
\usepackage{mathrsfs}
\usepackage{cmmib57}
\usepackage{multirow}
\usepackage[all]{xy}
\usepackage{longtable}
\usepackage[colorlinks,linkcolor=blue,anchorcolor=blue,citecolor=green,backref=page]{hyperref}
\usepackage{wrapfig}
\theoremstyle{plain}
    \newtheorem{thm}{Theorem}[section]
    \renewcommand{\thethm}
    {\arabic{section}.\arabic{thm}}
    \newtheorem{claim}[thm]{Claim}
     \newtheorem{conjecture}[thm]{Conjecture}
    \newtheorem{corollary}[thm]{Corollary}
    
    \newtheorem{lemma}[thm]{Lemma}
    \newtheorem{proposition}[thm]{Proposition}
    
    \newtheorem{theorem}[thm]{Theorem}

\theoremstyle{definition}
    \newtheorem{definition}[thm]{Definition}
    
    \newtheorem*{notation*}{Notation and Terminology}
    \newtheorem{remark}[thm]{Remark}
    \newtheorem{convention}[thm]{Convention}
\theoremstyle{remark}

\newenvironment{proof1}[1][Proof]{\par\noindent{\itshape #1. }}{\hfill$\square$\par}

\newcommand{\PP}{\mathbb{P}}

\newcommand{\Q}{\mathbb{Q}}

\newcommand{\alb}{\operatorname{alb}}

\newcommand{\id}{\operatorname{id}}

\newcommand{\mult}{\operatorname{mult}}

\newcommand{\Nlc}{\operatorname{Nlc}}
\newcommand{\Nklt}{\operatorname{Nklt}}

\newcommand{\Sing}{\operatorname{Sing}}

\newcommand{\Supp}{\operatorname{Supp}}

\newcommand{\Codim}{\operatorname{codim}}

\newcommand{\N}{\operatorname{N}}

\newcommand{\Pic}{\operatorname{Pic}}

\title{Log Calabi-Yau structure of algebraic varieties admitting a polarized endomorphism}

\author{Wentao Chang}
\address{School of Mathematical Sciences, Fudan University, 
People's Republic of China}
\email{wtchang21@m.fudan.edu.cn}

\author{De-Qi Zhang}
\address{Department of Mathematics, National University of Singapore, 
Republic of Singapore}
\email{matzdq@nus.edu.sg}
\makeatletter
\@namedef{subjclassname@2020}{
  \textup{2020} Mathematics Subject Classification}
\makeatother

\begin{document}

\begin{abstract}
Let $X$ be a normal projective variety admitting a polarized endomorphism $f$, i.e., $f^*H\sim qH$ for some ample divisor $H$ and integer $q>1$. Then Broustet and Gongyo proposed the conjecture that $X$ is of Calabi-Yau type (CY for short), i.e., $(X,\Delta)$ is lc for some effective $\Q$-divisor $\Delta$ and $K_X+\Delta\sim_{\Q}0$. We prove the conjecture when $X$ is a Gorenstein terminal 3-fold, extending 
the result of Sheng Meng for smooth threefolds.
We then study the singularity type and CY property for $(X,\Delta+\frac{R_{\Delta}}{q-1})$ when $(X,\Delta)$ is an $f$-pair, i.e., $K_X+\Delta=f^*(K_X+\Delta)+R_\Delta$ with $\Delta, R_{\Delta}$ being effective. 
In particular, we show:
(1) $K_X + \frac{R_f}{q-1}$ is $\Q$-Cartier and numerically trivial when $X$ is a $\Q$-factorial (or of klt type) $3$-fold; (2) $(X, \frac{R_{f}}{q-1})$ is log Calabi-Yau when $X$ is a surface with the Picard number $\rho(X)>1$ or $f^{-s}(P)=P$ for some prime divisor $P$ and $s>0$.
%In particular, when $X$ is $\Q$-factorial (or of klt type) and $\dim X \le 3$, we show that $f^*K_X-qK_X$, or equivalently $K_X + \frac{R_f}{q-1}$ is $\Q$-Cartier and numerically trivial.
\end{abstract}

\subjclass[2020]{
08A35, %Automorphisms, endomorphisms
14J17, %Singularities
32H50, %Iteration problems
%14E22: Ramification problems
%14J32: Calabi-Yau manifolds, mirror symmetry
14E30 %Minimal model program
}
\maketitle

\tableofcontents

\section{Introduction}
We work over an algebraically closed field $k$ of characteristic \(0\). 

%By a variety (resp. curve or surface), we mean a projective variety (resp. curve or surface).
Let \(f \colon X \to X\) be a surjective endomorphism of a normal projective variety with $\deg f>1$. Then $f$ is a finite morphism since every ample divisor on $X$ has an $f^*$-preimage. In the curve case, by the Hurwitz formula, \(X\) is either a rational curve or an elliptic curve. This is equivalent to saying that the anticanonical divisor \(-K_{X}\) is effective. In higher dimensions, the situation is more complicated. We focus on endomorphisms of the following type:
$f$ is {\it $q$-polarized} (or simply {\it polarized} ), i.e., $f^*H\sim qH$ for some ample Cartier divisor $H$ and integer $q>1$. In this case, $\deg f = q^{\dim X}$ by the projection formula.
A curve endomorphism $f$ is polarized if and only if deg $f>1$.

Let $f$ be a polarized endomorphism.  Meng and Zhang \cite[Theorem 1.5] {MZ22} showed that \(-K_{X}\) is effective when \(X\) is \(\mathbb{Q}\)-Gorenstein, which generalizes the curve case. 
%In general, $-K_X$ is weakly numerically equivalent to a effective $\Q$-Weil divisor Meng \cite[Theorem 1.5] {M20}. 
On the other hand, $X$ has mild singularity if it admits a polarized endomorphism. Indeed, 
Broustet and H\"oring \cite[Corollary 1.3]{BH14} proved that $X$ is log canonical when $X$ is  $\Q$-Gorenstein. 

Broustet and Gongyo proposed the following conjecture \cite[Conjecture 1.2]{BG17}.

\begin{conjecture}\label{conj:LCY}
Let $X$ be a normal projective variety admitting a polarized endomorphism. Then $X$ is of Calabi-Yau type.
\end{conjecture}
Recall that $X$ is of {\it Calabi-Yau type} if there exists an effective $\Q$-divisor $\Delta$ such that $(X,\Delta)$ is {\it log Calabi-Yau}, i.e., $K_X+\Delta\sim_{\Q}0$ and the pair $(X,\Delta)$ is log canonical.

%Conjecture \ref{conj:LCY} obviously implies the above results. ??
Conjecture \ref{conj:LCY} was confirmed by Broustet and Gongyo \cite[Theorem 1.3]{BG17} for surfaces, by Meng \cite[Theorem 1.2]{M23} for smooth threefolds, and by Yoshikawa \cite[Corollary]{Y21} for rationally connected smooth projective varieties. 

In this paper, we generalize \cite[Theorem 1.2]{M23} as follows.

\begin{thm}\label{thm: lcy for Gorenstein 3-fold}Let $X$ be a Gorenstein terminal 3-fold, admitting a $q$-polarized endomorphism $f$, with $q>1$. Then $X$ is of log Calabi-Yau type.
\end{thm}

For a finite surjective morphism $f: X \to X$ of a normal variety $X$, we can define 
$f^*D$ for any Weil integral divisor $D$ by restricting $f$ and $D$ to the regular locus of $X$, pulling back, and then taking the Zariski closure.

We have the ramification divisor formula
$$K_X = f^*K_X + R_f$$
with $R_f \ge 0$ the {\it ramification divisor.}

There is a related conjecture proposed by Gongyo \cite[Conjecture 1.3]{M23}.

\begin{conjecture}\label{conj Gongyo}
     Let $f: X \to X$ be a $q$-polarized endomorphism of a smooth projective variety. Then, after iteration of $f$, the pair $(X,\frac{R_f}{q-1})$ is log canonical.
 \end{conjecture}

Meng \cite[Theorems 4.1-4.2]{M23} proved Conjecture \ref{conj Gongyo}
for smooth surfaces except possibly when $X\cong \PP^2$ and $T_f=0$. Here $T_f$ is 
%{\it the total invariant divisor set}, i.e., 
the sum of all the $f^{-1}$-periodic prime divisors. 

We are going to introduce a more general conjecture. First, recall \cite[Definition 3.1]{Y21}.

\begin{definition}\label{def:f-pair}
    Let $f: X \to X$ be a finite surjective morphism of a normal projective variety, and $\Delta$ an effective Weil $\Q$-divisor. The pair $(X,\Delta)$ is called an {\it $f$-pair} if 
    $$R_{\Delta}:=\Delta+R_f-f^*\Delta\geq 0$$ (without assuming $K_X+\Delta$ to be $\Q$-Cartier here), or equivalently (all pullbacks being for Weil divisors):
$(*) \,\, K_X + \Delta = f^*(K_X + \Delta) + R_{\Delta}$. 
\end{definition}
An $f$-pair is also an $f^s$-pair for any $s \ge 1$ by iterating the $(*)$.
Since the ramification divisor $R_f \ge 0$, the pair $(X, 0)$ is always an $f$-pair.
More generally, if $D$ is {\it totally invariant} for some reduced divisor $D$, i.e., $f^{-1}(D) = D$, then the log ramification divisor formula as in \cite[Theorem 11.5]{I82} shows that $K_X + D = f^*(K_X + D) + R_D$
with $R_D \ge 0$, so the pair $(X, D)$ is an $f$-pair.

\begin{conjecture} \label{conj Delta}
     Let $f: X \to X$ be a $q$-polarized endomorphism on a normal projective variety with $\Delta$ an effective $\Q$-divisor such that $(X,\Delta)$ is an $f$-pair. Then, after iteration of $f$, the following hold.
     \begin{enumerate}
         \item $K_X+\Delta+\frac{R_{\Delta}}{q-1}$ is $\Q$-Cartier and numerically trivial.
         \item $(X,\Delta + \frac{R_{\Delta}}{q-1})$ is log canonical.
     \end{enumerate} 
 \end{conjecture}

 \begin{remark}\label{rem:equiv conj}
 $ $
 \begin{enumerate}
%     \item 
%      Conjecture \ref{conj Delta} holds for curve pairs; see Proof of Theorem ??
      \item
   Conjecture \ref{conj Delta} implies Gongyo 
     Conjecture \ref{conj Gongyo} by taking $\Delta=0$.
   
   \item 
   %Conjecture \ref{conj Delta} (1) is equivalent to 
  As in Proposition \ref{prop:mori-fiber-lcy} $(*)$, $f^*(K_X+\Delta) - q(K_X+\Delta) \sim_{\Q} (q-1)(K_X+\Delta+\frac{R_{\Delta}}{q-1})$. Conjecture \ref{conj Delta} (1) implies $f^*K_X - qK_X \equiv 0$ (applied to the $f$-pair $(X, 0)$).

   \item 
   Conjecture \ref{conj Delta} is equivalent to $(X,\Delta + \frac{R_{\Delta}}{q-1})$ being log Calabi-Yau by the abundance \cite[Theorem 1.2]{G13}. In particular,  
   $f^*(K_X+\Delta)\sim_{\Q}q(K_X+\Delta)$ from (2) above.
   \end{enumerate}
 \end{remark}

As a first step toward Conjecture \ref{conj Delta}, we prove the following two results. 

\begin{theorem}[Part of Theorem \ref{Delta leq 1}]\label{thm:Delta+R/q-1 leq 1}
Let $f: X \to X$ be a $q$-polarized endomorphism of a normal projective variety and $\Delta$ an effective $\Q$-divisor such that $(X,\Delta)$ is an $f$-pair. Then, after iteration of $f$, $\Delta + \frac{R_{\Delta}}{q-1} \leq 1$, i.e., all the coefficients of the divisor are $\le 1$. In particular, Conjecture \ref{conj Delta} holds for curve pairs. 
\end{theorem}

 \begin{theorem}\label{thm:(X,Delta) lc}
Let $f: X \to X$ be an int-amplified endomorphism (cf.~Convention \ref{notation}(5))
of a normal projective variety, and $\Delta$ an effective $\Q$-divisor such that $(X,\Delta)$ is an $f$-pair and $K_X+\Delta$ is $\Q$-Cartier. Then $(X,\Delta)$ is log canonical.
\end{theorem}

\begin{remark}
    Broustet-H\"oring \cite[Theorem 1.4]{BH14} proved Theorem \ref{thm:(X,Delta) lc}, assuming $f$ is polarized and $\Delta$ is $f^{-1}$-invariant and reduced. Our proof follows their idea. After finishing the proof of Theorem \ref{thm:(X,Delta) lc}, we noticed that they remarked (though with no proof given) ``Our proof (of \cite[Theorem 1.4]{BH14})  actually works more generally for log pairs $(X,\Delta)$ such that $K_X + \Delta$ is $\Q$-Cartier and a logarithmic ramification formula holds (for polarized $f$).''

%claimed the general case (with no restriction on $\Delta$) is true, though omit the proof.
\end{remark}

\begin{corollary}\label{cor:quasi-'etale lcy}
In the setting of Theorem \ref{thm:(X,Delta) lc}, if $R_{\Delta}=0 $, then $(X,\Delta)$ is log Calabi-Yau.
\end{corollary}

%From the point of view of adjunction, 
We have two methods of induction for Conjecture \ref{conj Delta}. The first is by inversion of adjunction.

\begin{theorem}\label{thm:polarized-inversion-of-adjunction}
    Let $f: X \to X$ be a $q$-polarized endomorphism of a normal projective variety and $\Delta$ an effective $\Q$-divisor such that $K_X+\Delta$ is $\Q$-Cartier and $(X,\Delta)$ is an $f$-pair. Let $Z$ be a log canonical center of $(X,\Delta)$ satisfying $f^{-1}(Z)=Z$. Let $(K_X+\Delta)|_{\widetilde{Z}}\sim_{\Q}K_{\widetilde{Z}}+\Delta_{\widetilde{Z}}+M_{\widetilde{Z}}$ be the subadjunction formula as in Definition \ref{def:subadjunction}, where $\widetilde{Z}$ is the normalization of $Z$. Then we have:
    \begin{enumerate}
        \item
$(\widetilde{Z},\Delta_{\widetilde{Z}})$ is an $f|_{\widetilde{Z}}$-pair (and $f|_{\widetilde{Z}}$ is $q$-polarized), with $R_{\Delta_{\widetilde{Z}}}=R_{\Delta}|_{\widetilde{Z}}$.
\item
$Z$ is also an lc center of $(X,\Delta+\frac{R_{\Delta}}{q-1})$. Suppose $M_{\widetilde{Z}}$ is $\Q$-Cartier. Then $M_{\widetilde{Z}}\sim_{\Q} 0$. Further, $(K_{\widetilde{Z}},\Delta_{\widetilde{Z}}+\frac{R_{\Delta}|_{\widetilde{Z}}}{q-1})$ is log canonical if and only if $(X,\Delta+\frac{R_{\Delta}}{q-1})$ is lc near $Z$.
    \end{enumerate}
\end{theorem}

\begin{remark}
    Recently, Bakker, Filipazzi, Mauri and Tsimerman \cite[Theorem 1.5]{BFMT25} announced a proof of the b-semiampleness conjecture of Prokhorov--Shokurov, which allows us to remove the condition ``$M_{\widetilde{Z}}$ is $\Q$-Cartier" in Theorem \ref{thm:polarized-inversion-of-adjunction}. Indeed, if $M_{\widetilde{Z}}$ is the push forward of a semiample divisor (thus effective), then $M_{\widetilde{Z}}\sim_{\Q} 0$ since $M_{\widetilde{Z}}\equiv 0$ (weakly numerically) as in Theorem \ref{thm:subadj-ramification}.
\end{remark}

%For klt pair, we provide an induction method to reduce Conjecture \ref{conj Delta} to Picard number $\rho(X)=1$ case. 

%Theorem \ref{thm:klt to Picard 1} below 
The second method is to run log MMP and reduce Conjecture \ref{conj Delta} (for klt pairs) to the Picard number one case. 

 \begin{theorem}\label{thm:klt to Picard 1}
    Suppose that Conjecture \ref{conj Delta} holds for all klt pairs $(X, \Delta)$ with Picard number $\rho(X) = 1$. Then Conjecture \ref{conj Delta} holds for all klt pairs
    $(X,\Delta)$.
\end{theorem}

As a consequence of the arguments in the proof, we will also prove the following result.

\begin{corollary}\label{cor:numerically trivial induction}
%(See Remark \ref{rem:equiv conj})
\mbox{}
\begin{enumerate}
\item 
Suppose that Conjecture \ref{conj Delta} holds for all klt pairs $(X, \Delta)$ with $\dim X\leq{n-2}$. Then Conjecture \ref{conj Delta} (1) holds for all klt pairs $(X, \Delta)$ with $\dim X\leq n$. 
%\item
%Suppose $\dim X\leq 3$. Then Conjecture \ref{conj Delta} (1) holds for all klt pairs $(X, \Delta)$. Namely, $f^*(K_X+\Delta) -q(K_X+\Delta) $ is $\Q$-Cartier and numerically trivial.
    %by Theorem \ref{thm:Delta+R/q-1 leq 1}.
\item
Conjecture \ref{conj Delta}  (1) holds for all $\Q$-factorial (or of klt type) variety with $\dim X\leq 3$.
%Suppose $f: X \to X$ is a $q$-polarized endomorphism of a $\Q$-factorial (or of klt type) normal projective variety, and $\dim X \le 3$. If $(X, \Delta)$ is an $f$-pair then, after iteration of $f$, $f^*(K_X+\Delta) -q(K_X+\Delta)$, or equivalently $K_X+\Delta+\frac{R_{\Delta}}{q-1}$, is $\Q$-Cartier and numerically trivial. 
%In particular, $f^*K_X-qK_X$ (or equivalently, $K_X + \frac{R_f}{q-1})$) is $\Q$-Cartier and numerically trivial.
\end{enumerate}
\end{corollary}

Applying Corollary \ref{cor:numerically trivial induction} to the $f$-pair $(X, 0)$, we obtain (cf. Remark \ref{rem:equiv conj}):
\begin{theorem}
Let $f: X \to X$ be a $q$-polarized endomorphism of a $\Q$-factorial (or of klt type) normal projective variety, and $\dim X \le 3$.
Then, after iteration of $f$, $f^*K_X-qK_X$, or equivalently $K_X + \frac{R_f}{q-1}$, is $\Q$-Cartier and numerically trivial.
\end{theorem}

Applying the methods above, we generalize \cite[Theorems 4.1-4.2]{M23} to the singular case.

\begin{theorem}\label{thm: conj delta surface}
Conjecture \ref{conj Delta} holds for surface pairs satisfying either $(1)$ or $(2)$ below.
    \begin{enumerate}
        \item The Picard number $\rho(X)\geq 2$.

        \item 
        The sum $T_f$ of all the $f^{-1}$-periodic prime divisors is nonzero, i.e., 
        $f^{-s}(P)=P$ for some prime divisor $P$ and $s>0$.
    \end{enumerate}
\end{theorem}

Section \ref{sec:preliminaries} is preliminary. We prove Theorem \ref{thm: lcy for Gorenstein 3-fold}, Theorems \ref{thm:Delta+R/q-1 leq 1} - \ref{thm:(X,Delta) lc} and Corollary \ref{cor:quasi-'etale lcy}, Theorem \ref{thm:polarized-inversion-of-adjunction}, Theorem \ref{thm:klt to Picard 1} and Corollary \ref{cor:numerically trivial induction}, and Theorem \ref{thm: conj delta surface} in Sections \ref{sec:fano type fibration over abelian}, \ref{sec:singularity of (X,Delta+R/q-1)}, \ref{sec:ramification-subadj}, \ref{sec:klt to picard 1} and \ref{section: surface appl} respectively.
We give Appendix \ref{appen:equi-Q-fac-model} on equivariant $\Q$-factorial models, supplementing a bit to Moraga-Y\'a\~nez-Yeong \cite[Theorem 1.2]{MYY24}.

\par \vskip 1pc \noindent
{\bf Acknowledgment.}
The first author would like to thank his advisor Professor Meng Chen for the constant encouragement and support, the National University of Singapore for the hospitality during his visit, and China Scholarship Council (No.202406100005) for the support.
%This work was completed during a visit of the first author at National University of Singapore. 
%He would like to thank Professor De-Qi Zhang for hospitality and helpful discussions and suggestions. The first author is supported by China Scholarship Council (No.202406100005). 
The second author is supported by the ARF A-8002487-00-00 of NUS.

\section{Preliminaries}\label{sec:preliminaries}
We use the following notation and terminology throughout the paper.
\begin{convention}\label{notation}
     Let $X$ be a projective variety of dimension $n \ge 1$.
    \begin{enumerate}
    \item The symbols \(\sim\) (resp.\ \(\sim_{\mathbb{Q}}\), \(\equiv\)) denote the \textit{linear equivalence} (resp.\ \(\mathbb{Q}\)\textit{-linear equivalence}, \textit{numerical equivalence}) on \(\mathbb{Q}\)- (or \(\mathbb{R}\)-) Cartier divisors. 
    Weil divisors $D_1$ and $D_2$ on normal projective $X$, 
are \textit{weakly numerically equivalent}, denoted as $D_1 \equiv_w D_2$, if $(D_1-D_2) . L_1 \dots L_{n-1} = 0$ for all Cartier divisors $L_j$.
    \item Denote by \(\operatorname{NS}(X)=\operatorname{Pic}(X)/\operatorname{Pic}^{0}(X)\) the \textit{Neron-Severi group} of \(X\). Let \(\operatorname{N}^{1}(X):=\operatorname{NS}(X)\otimes_{\mathbb{Z}}\mathbb{R}\) be the space of \(\mathbb{R}\)-Cartier divisors modulo numerical equivalence and \(\rho(X):=\dim_{\mathbb{R}}\operatorname{N}^{1}(X)\) the \textit{Picard number} of \(X\). Denote by \(\operatorname{N}_{n-1}(X)\) the space of $(n-1)$-cycles  modulo weakly numerical equivalence.

    \item Denote by $\mu_P B$ the coefficient of a prime divisor $P$ in a Weil $\Q$-divisor $B$ when $X$ is normal.
    
    \item Let \(f:X\to X\) be a surjective endomorphism. A subset $S$ of $X$ is called {\it $f$-invariant} (resp. {\it $f^{-1}$-invariant}), if $f(S)=S$ (resp. $f^{-1}(S)=S$). $S$ is {\it $f$-periodic} (resp. {\it $f^{-1}$-periodic}) if $f^k(S)=S$ (resp. $f^{-k}(S)=S$) for some positive integer $k$. 
    
    \item \label{pol-int} $f:X\to X$ is {\it $q$-polarized} (or simply {\it polarized}) if there is an ample Cartier divisor $H$ such that $f^*H\sim qH$ for some integer $q > 1$. 
    $f$ is {\it int-amplified} if there is an ample Cartier divisor $H$ such that $f^*H-H$ is ample.

    \item A projective mophism $f:X\to Y$ between normal varieties is a {\it fibration} if $f_*\mathcal{O}_X=\mathcal{O}_Y$.
    
    \item A finite morphism $f$ between varieties is {\it quasi-\'etale} if $f$ is \'etale in codimension 1.

    \item $X$ is $\Q$-{\it abelian} if $X$ is normal and there is a quasi-\'etale surjective morphism $A \to X$ from an abelian variety $A$.

\end{enumerate}
\end{convention}

\subsection{Pair and singularity}
We will use standard notation and results of minimal model program (MMP) in \cite{KM98}.
We say $(X,\Delta)$ is a {\it sub-pair} if $X$ is a normal variety and $\Delta$ is a Weil $\Q$-divisor on $X$ such that $K_X + \Delta$ is a $\Q$-Cartier divisor. We say $(X,\Delta)$ is a {\it pair} if $\Delta$ is further effective.
Let $\pi:\widetilde{X}\to X$ be a proper birational morphism from a normal variety $\widetilde{X}$. Then we can uniquely write
$K_{\widetilde{X}}+\widetilde{\Delta}=\pi^*(K_X+\Delta)$,
such that $\pi_*\widetilde{\Delta}=\Delta$. 

Let $E$ be a prime divisor on $\widetilde{X}$, the {\it discrepancy} of $E$ with respect to $(X,\Delta)$ is defined by $a(E,X,\Delta)=-\mu_E\widetilde{\Delta}$. A pair (resp. sub-pair) $(X,\Delta)$ is {\it log canonical} , or {\it lc} (resp. {\it sub-lc}) for short, if $a(E,X,\Delta)\geq-1$ for every prime divisor $E$ over $X$. A pair (resp. sub-pair) $(X,\Delta)$ is {\it Kawamata log terminal} or {\it klt} for short (resp. {\it sub-klt}) if $a(E,X,\Delta)> -1$ for every prime divisor $E$ over $X$. A normal variety $X$ is {\it canonical} (resp. {\it terminal}) if $(X,0)$ is a pair and $a(E,X,0)\geq 0$ (resp.$>0$), and is of {\it klt type} if $(X,\Delta)$ is klt for some effective divisor $\Delta$ on $X$. 

 A {\it non-lc place} (resp. {\it non-klt place}) of $(X,\Delta)$ is a prime divisor $E$
 over $X$ such that $a(E,X,\Delta)<-1$ (resp. $a(E,X,\Delta)\leq -1$). A {\it non-lc center} (resp. {\it non-klt center}) is the image on $X$ of a non-lc (resp.  non-klt) place. Let $\Nlc(X,\Delta)$ (resp. $\Nklt(X,\Delta)$) be the {\it non-lc} (resp. {\it non-klt}) {\it locus} of $(X,\Delta)$, which is the union of its non-lc (resp. non-klt) centers. By \cite[Corollary 2.31]{KM98}, $\Nlc(X,\Delta)$ and $\Nklt(X,\Delta)$ are closed subsets of $X$.

A pair $(X,\Delta)$ is {\it log Calabi-Yau} if $K_X+\Delta\sim_{\Q} 0$ and $(X,\Delta)$ is log canonical. $X$ is of {\it Calabi-Yau type} if $(X,\Delta)$ is log Calabi-Yau for an effective Weil $\Q$-divisor $\Delta$. $X$ is of {\it Fano type} if $(X,\Delta)$ is klt and $-(K_X+\Delta)$ is ample for an effective Weil $\Q$-divisor. 
\begin{lemma}\label{lemma:quasi etale or birational lcy}
    Let $\pi:X\dashrightarrow Y$ be a quasi-\'etale morphism or birational map between normal projective varieties. 
    \begin{enumerate}

    \item Suppose $\pi$ is quasi-\'etale. Then $X$ is of Calabi-Yau type if and only if so is $Y$.

    \item Suppose $\pi$ is small birational. Then $X$ is of Fano Type if and only if so is $Y$. 

    \item Suppose $\pi$ is small birational and $\Delta$ is an effective Weil $\Q$-divisor on $X$. Then $(X,\Delta)$ is log Calabi-Yau if and only if so is $(Y,\pi_*\Delta)$. 

    \item Suppose $\pi$ is a birational morphism and $Y$ has only rational (or klt type) singularity. Then $\pi^*:\Pic^0(X)\to \Pic^0(Y)$ is an isomorphism. In particular, 
   if $D$ is $\Q$-Cartier and numerically trivial, so is $\pi_*D$.

    \end{enumerate}
    \end{lemma}
\begin{proof}
     (1) is from \cite[Lemma 2.3]{M23}. (2) is from \cite[Proposition 2.4]{Y21}. For (3), we may assume $(K_X+\Delta)$ is log Calabi-Yau. Taking a common resolution $W$ of $X,Y$ with $p:W\to X$ and $q:W\to Y$, it is enough to note that $(K_Y+\pi_*\Delta)=q_*p^*(K_X+\Delta)\sim_{\Q}0$, and $p^*(K_X+\Delta)=q^*(K_Y+\pi_*\Delta)$ from the negativity lemma \cite[Lemma 3.39]{KM98} (applied to $q$). (4) is from \cite[Proposition 2.3]{R83}.
     %Reid's proof of his result (see e.g. \cite[Lemma 5.1]{MZ18}).
\end{proof}

The following well-known lemma will be used.
\begin{lemma}\label{lemma:B.H^(n-1)=0}
    Let $B$ and $H$ be $\Q$-Cartier divisors on a normal projective variety $X$ of dimension $n$. Suppose $B$ is pseudo-effective, $H$ is ample and $BH^{n-1}=0$. Then $B\equiv 0$ .

\end{lemma}

\begin{proof}
    For ample divisors $H_1$,...,$H_{n-1}$, we can choose some $\epsilon>0$ sufficiently small, such that $H-\epsilon H_i$ is ample for $1\leq i\leq n-1$. Then $B$ being pseudo-effective implies
$0\leq B(\epsilon H_1)\dots(\epsilon H_{n-1})\leq BH^{n-1}=0$,
    which implies $BH_1\dots H_{n-1}=0$.

    For any Cartier divisor $D_i$, write $D_i\sim H_i'-H_i'' $ for some ample divisors $H_i'$ and $H_i''$. Then
    %\[
   $BD_1\dots D_{n-1}=B(H_1'-H_1'')\dots (H_{n-1}'-H_{n-1}'')=0$.
    %\]
    Hence $B$ is numerically trivial by \cite[Lemma 2.3]{MZ18}.
     %$B\equiv 0$.
\end{proof}

\subsection{Generalized pair, canonical bundle formula and adjunction}
First, we recall the definition regarding b-divisors.
    Let $X$ be a normal variety. A {\it b-$\Q$-Cartier b-divisor} over $X$ is the choice of a projective birational morphism \(Y\to X\) from a normal variety and a \(\mathbb{Q}\)-Cartier $\Q$-divisor \(M\) on \(Y\) up to the following equivalence: another projective birational morphism \(Y^{\prime}\to X\) from a normal variety and a \(\mathbb{Q}\)-Cartier divisor \(M^{\prime}\) defines the same b-\(\mathbb{Q}\)-Cartier b-divisor if there is a common resolution \(W\to Y\) and \(W\to Y^{\prime}\) on which the pullbacks of \(M\) and \(M^{\prime}\) coincide.

\begin{definition} [{cf. \cite[Definition 1.4]{BZ16}}]\label{def:generalized-pair}
    A {\it generalized} pair consists of
\begin{enumerate}
    \item a normal variety \(X\),
    \item a \(\mathbb{Q}\)-divisor \(B\geq0\) on \(X\), and
    \item a b-\(\mathbb{Q}\)-Cartier b-divisor over \(X\) represented by some projective birational morphism \(X' \xrightarrow{\phi} X\) and \(\mathbb{Q}\)-Cartier \(\mathbb{Q}\)-divisor \(M'\) on \(X'\),
\end{enumerate}
such that \(M'\) is nef and \(K_X + B + M\) is \(\mathbb{Q}\)-Cartier, where \(M := \phi_* M'\). Here \(M'\) is called the {\it nef part} of the pair.
\end{definition}
Now we define generalized singularities. Let $\pi:X_1\to X$ be a proper birational morphism from a normal variety $X_1$. Then we can uniquely write
\[
K_{X_1} + B_1 + M_1 = \pi^*(K_{X} + B + M)
\]
such that $\pi_*B_1=B$. For a prime divisor \(D\) on \(X'\), the {\it generalized discrepancy} \(a(D, X, B + M)\) is defined to be \(-\mu_D B_1\). A generalized pair \((X, B + M)\) is {\it generalized lc} (resp. {\it generalized klt})  if for each prime divisor \(D\) over $X$, the generalized discrepancy \(a(D, X, B + M)\) is \(\geq -1\) (resp. \(> -1\)) . Note that \((X, B + M)\) is generalized lc (resp. generalized klt) implies $(X,B)$ is lc (resp. klt) if $M$
is $\Q$-Cartier.
\begin{lemma}\label{lemma:moduli part M=0}
     Let $(X,B+M)$ be a generalized pair such that $M$ is $\Q$-Cartier and $M\equiv 0$. Then $(X,B+M)$ is generalized lc (resp. klt) if and only if $(X,B)$ is lc (resp. klt).
\end{lemma}

\begin{proof}
    For any prime divisor $D$ over $X$, take an birational model $\pi:X_1\to X$ such that $D$ is a divisor on $X_1$ and the moduli part $M_1$ is nef on $X_1$. Write $K_{X_1} + B_1 + M_1 = \pi^*(K_{X} + B + M)$.
    Note that $M_1=\pi^*M+E\equiv E$ for some $\pi$-exceptional divisor, we obtain $E= 0$ by the negativity lemma \cite[Lemma 3.39]{KM98}. Hence $M_1=\pi^*M$ and $K_{X_1} + B_1= \pi^*(K_{X} + B)$, which implies
    $a(D,X,B+M)=a(D,X,B)$. %The conclusion then follows.
\end{proof}

\begin{definition}[canonical bundle formula]\label{def:canonical bundle formula}
    Let $(X,B)$ be a sub-pair, and $f:X\to Z$ a fibration to a normal variety $Z$. Suppose that $(X,B)$ is sub-lc over the generic point of $f$ and $K_X+B\sim_{f,\Q}0$. 

    For each prime divisor \(D\) on \(Z\), let \(t_{D}\) be the lc threshold of \(f^{*}D\) with respect to \((X,B)\) over the generic point of \(D\), i.e., \(t_{D}\) is the largest number such that \((X,B+t_{D}f^{*}D)\) is sub-lc over the generic point of \(D\).  Let \(b_{D}=1-t_{D}\), and define \(B_{Z}=\sum b_{D}D\) where the sum runs over all the prime divisors on \(Z\). Note that $B_Z$ is effective if $B$ is effective.

    By assumption, \(K_X+B\sim_{\mathbb{Q}}f^*L_Z\) for some \(\mathbb{Q}\)-Cartier \(\mathbb{Q}\)-divisor \(L_Z\) on \(Z\). Letting \(M_Z=L_Z-(K_Z+B_Z)\) we get
\[
K_X+B\sim_{\mathbb{Q}}f^*(K_Z+B_Z+M_Z).
\]
We call \(B_Z\) the \textit{discriminant part} and \(M_Z\) the \textit{moduli part} of adjunction. Obviously \(B_Z\) is uniquely determined, but \(M_Z\) is determined only up to \(\mathbb{Q}\)-linear equivalence because it depends on the choice of \(L_Z\).
\end{definition}

\begin{theorem}[{\cite[Section 3.4]{B19}}]\label{thm:M b-nef}
    Suppose $\Delta$ is effective. Then $(Z,B_Z+M_Z)$ carries a structure of generalized pair, and the pair $(X,\Delta)$ is lc if and only if $(Z,B_Z+M_Z)$ is generalized lc.
\end{theorem}

\begin{definition}[log canonical center and subadjunction]\label{def:subadjunction}
Let $(X,\Delta)$ be a sub-pair. We say a closed subvariety
 $Z\subseteq X$ is a {\it log canonical center} (or {\it lc center} for short) if 
 \[a(Z,X,\Delta)=\inf\{a(E,X,\Delta) \, | \, \text{center}_XE=Z,E\text{ is a prime divisor over }X \}=-1.\]
 By \cite[Corollary 2.31]{KM98}, the infimum can be obtained by some prime divisor $E$ over $X$, then such a divisor is called a {\it log canonical place} (or {\it lc place} for short) of $(X,\Delta)$. 
 %Denote by $\text{LC}(X,\Delta)$ the union of lc centers of $(X,\Delta)$. Then $\text{LC}(X,\Delta)=\overline{\Nklt(X,\Delta)-\Nlc(X,\Delta)}$.
 %Note that the pair $(X,\Delta)$ is log canonical in a neighborhood of the generic point of 
 %$Z$. 

 Now we apply subadjunction as in \cite{FH22} and \cite{FH23}.
 Suppose $\Delta$ is effective.
 Let $T$ be a prime divisor over $X$ such that $a(T,X,\Delta)=-1$. 
 Let $Z$ be the center of $T$ on $X$ , and $\widetilde{Z}\to Z$ the normalization. Take a log resolution $\pi:X_1\to X $ of $(X,\Delta)$ such that $T$ is a divisor on $X_1$, which induces a morphism $\pi_T:T\to \widetilde{Z}$.
 Let $\Delta_T=(\Delta_{X_1}-T)|_T$, where $K_{X_1}+\Delta_{X_1}=\pi^*(K_X+\Delta)$. Since $(T,\Delta_T)$ is lc over the generic point of $\widetilde{Z}$, and $K_T+\Delta_T\sim_{\pi_T,\Q}0$, we have the canonical bundle formula 
 \[K_T+\Delta_T\sim_{\Q}\pi_T^*(K_{\widetilde{Z}}+\Delta_{\widetilde{Z}}+M_{\widetilde{Z}}),\]
which gives the subadjunction formula via $Z$:
\[(K_X+\Delta)|_{\widetilde{Z}}\sim_{\Q} K_{\widetilde{Z}}+\Delta_{\widetilde{Z}}+M_{\widetilde{Z}}.\]
\begin{enumerate}
    \item \cite[theorem 1.2]{FH22}, $\Delta_{\widetilde{Z}}$ is independent of the choice of log resolution and prime divisor $T$. And $(\widetilde{Z},\Delta_{\widetilde{Z}}+M_{\widetilde{Z}})$  is a generalized pair as defined in Definition \ref{def:generalized-pair}.

    \item \cite[theorem 1.1]{FH23}, the pair $(X,\Delta)$ is lc near $Z$ if and only if $(\widetilde{Z},\Delta_{\widetilde{Z}}+M_{\widetilde{Z}})$ is generalized lc.

    \item When $\Codim_X Z=1$, we have $M_{\widetilde{Z}}=0$ and  $(K_X+\Delta)|_{\widetilde{Z}}= K_{\widetilde{Z}}+\Delta_{\widetilde{Z}}$ is the classical subadjunction formula (cf.~\cite[Chapter 16]{K92},\cite{K07}). In this case, we can also define the subadjunction formula  for sub-pairs.
    
\end{enumerate}
\end{definition}

\subsection{Endomorphisms}
In this subsection, {\it $X$ is a normal projective variety, and $f: X \to X$ is a surjective endomorphism.}
%which is finite by considering $f^*$ action on $N^1_{\R}(X)$.
We list some results on polarized endomorphisms.

\begin{proposition} [{\cite[Theorem 3.11, Corollary 3.12]{MZ18}}]\label{prop:polarized basic property}
    \mbox{}
    \begin{enumerate}
    \item Let $Z$ be a closed subvariety on $X$ such that $f(Z)=Z$. If $f$ is $q$-polarized,
 then so is the restriction $f|_Z$.
    \item Let \(\pi\colon X \to Y\) be a surjective morphism between normal projective varieties, and \(g\colon Y \to Y\) a surjective endomorphism such that \(\pi \circ f = g \circ \pi\). If \(f\) is $q$-polarized, then so is \(g\).

    \item Let \(\pi\colon X \dashrightarrow Y\) be a dominant rational map between normal projective varieties of the same dimension, and \(g\colon Y \to Y\) a surjective endomorphism such that \(\pi \circ f = g \circ \pi\). Then \(f\) is $q$-polarized if and only if so is \(g\).
\end{enumerate}
\end{proposition}

Let $T_f$ be the sum of all the $f^{-1}$-periodic prime divisors,  and $T\leq T_f$ a reduced divisor. 
\begin{proposition}\label{prop:T_f}
    Suppose $f$ is polarized. Then:
    \begin{enumerate}
    \item $T_f$ has only finite many nonzero irreducible components.

    \item If $K_X + T$ is $\Q$-Cartier, then $(X, T)$ is log canonical.

    \item If $K_X+T_f$ is $\Q$-Cartier, then $-(K_X+T_f)$ is $\Q$-movable, i.e., for any prime divisor $P$ on $X$, $-(K_X+T_f)\sim_{\Q} D$ for an effective divisor $D$ such that $P\nsubseteq\Supp D$.
    
    \item Suppose $(X,\Delta)$ is an $f$-pair (see Definition \ref{def:f-pair}) and $K_X+\Delta$ is $\Q$-Cartier. Then $-(K_X+\Delta)\sim_{\Q} E$ for some effective $\Q$-divisor $E$. In particular,
    if $K_X +\Delta$ is pseudo-effective then $K_X+\Delta\sim_{\Q} 0$.
    \end{enumerate}
\end{proposition}

\begin{proof}
    (1) follows from \cite[Corollary 3.8]{MZ20}. (2) follows from \cite[Theorem 1.4]{BH14}. (3) follows from \cite[Theorem 2.6, arXiv ver]{M23}. (4) is from the proof of \cite[Theorem 1.5]{MZ22}, with $-K_X$  replaced by $-(K_X+\Delta)$. 
 \end{proof}

The following result implies the codimension one part of Conjecture \ref{conj Gongyo}. 

\begin{theorem}[{\cite[Theorem 3.1]{M23}}]\label{prop:Rf leq q-1}
    Suppose $f$ is \(q\)-polarized. Then, after iteration of $f$, 
    %the coefficient \(\mu_{P}\) of each prime divisor \(P\) in the ramification divisor $R_f$ has 
    \(\mu_{P}R_f \leq q - 1\) holds, and the equality holds if and only if \(P\) is \(f^{-1}\)-periodic.
\end{theorem}

Polarized or int-amplified endomorphism has many periodic points.

\begin{proposition}[{\cite[Propsition 5.1]{F03}}]\label{prop:zd of f-periodic pt}
    Suppose $f: X \to X$ is int-amplified. Then the $f$-periodic points are Zariski dense in $X$.
\end{proposition}

Meng and Zhang \cite{MZ18} proved that MMP is $f$-equivariant for a polarized endomorphism.

\begin{theorem}\label{thm:polarized-emmp}
    Suppose $X$ is $\Q$-factorial and $f$ is polarized. Then any finite sequence of log MMP starting from $X$ is $f$-equivariant after iteration of $f$.
\end{theorem}
\begin{proof}
    It follows from \cite[Lemma 6.4, 6.5 and Remark 6.3]{MZ18}.
\end{proof}

%\begin{theorem}\label{thm:polarized-emmp}
%    Let \(f: X \to X\) be a polarized endomorphism of a \(\mathbb{Q}\)-factorial klt projective variety \(X\). Then, replacing \(f\) by a positive power, there exist a \(Q\)-abelian variety \(Y\), a morphism \(X \to Y\), and an \(f\)-equivariant relative MMP over \(Y\)
%    \[
%    X = X_1 \dashrightarrow \cdots \dashrightarrow X_i \dashrightarrow \cdots \dashrightarrow X_r = Y
%    \]
%    (i.e., \(f = f_1\) descends to \(f_i\) on each \(X_i\)), with every \(X_i \dashrightarrow X_{i+1}\) a divisorial contraction, a flip or a Fano contraction, of a \(K_{X_i}\)-negative extremal ray, such that we have:
%    \begin{enumerate}
%        \item If \(K_X\) is pseudo-effective, then \(X = Y\) and it is \(Q\)-abelian.
        %(see Proposition 1.6 or Lemma 4.6 for the lifting of \(f\)).
%        \item If \(K_X\) is not pseudo-effective, then for each \(i\), \(X_i \to Y\) is equi-dimensional holomorphic with every fibre (irreducible) rationally connected and \(f_i\) is polarized by some ample Cartier divisor \(H_i\). The \(X_{r-1} \to X_r = Y\) is a Fano contraction.
%        \item \(f^*|_{\mathrm{N}^1(X)}\) is a scalar multiplication: \(f^*|_{\mathrm{N}^1(X)} = q \cdot \mathrm{id}\), if and only if so is \(f_r^*|_{\mathrm{N}^1(Y)}\).
%    \end{enumerate}
%\end{theorem}

%
%
%
%
\section{Fano type fibration over an abelian variety; Proof of Theorem \ref{thm: lcy for Gorenstein 3-fold}}
\label{sec:fano type fibration over abelian}

In this section, we treat equivariant Fano type fibration over an abelian variety, and prove Theorem \ref{thm: lcy for Gorenstein 3-fold}.
First, we recall Yoshikawa \cite[Theorem 6.6]{Y21}.

\begin{proposition}\label{prop:etale cover of f-pair}
Let $X$ be a $\Q$-factorial klt projective variety admitting an int-amplified endomorphism $f$. Then there exists a quasi-\'etale finite cover $\mu:\widetilde{X} \to X$ such that the Albanese morphism $\pi := \alb_{\widetilde{X}}: \widetilde{X}\to Y$ is a fibration and $\widetilde{X}$
is of Fano type over $Y$.
%Then the following morphisms and pairs
%\begin{center}
%\begin{tikzcd}
%(X,\Delta)
%&(\widetilde{X},\widetilde{\Delta})
%\arrow[l,"\mu"']
%\arrow[r,"\pi"]
%&Y,
%\end{tikzcd}
%\end{center}
%exist and satisfy

%$\bullet \quad$ $\mu$ is a quasi-\'etale cover of $(X,\Delta)$, i.e., %$K_{\widetilde{X}}+\widetilde{\Delta}=\mu^*(K_X+\Delta)$, 

%$\bullet \quad$  $\pi$ is a fiber space, i.e., $\pi_*\OO_{\widetilde{X}}=\OO_Y$,

%$\bullet \quad$ $Y$ is an abelian variety,

%$\bullet \quad$ $\widetilde{X}$ is a normal projective variety, and 

%$\bullet \quad$ $(\widetilde{X},\widetilde{\Delta})$ is of Fano type over Y.

%In particular, $\pi$ is the Albanese morphism. 
Moreover, if $\phi$ is a surjective endomorphism of $X$, then we obtain the following diagram

\begin{center}
\begin{tikzcd}
X\arrow[d,"\phi^m"]
&\widetilde{X}\arrow[l,"\mu"']\arrow[d,"\widetilde{\phi}"]\arrow[r,"\pi"]
&Y\arrow[d,"\phi_Y"]\\
X
&\widetilde{X}\arrow[l,"\mu"']\arrow[r,"\pi"]
&Y
\end{tikzcd}   
\end{center}

for some $m$, where $\widetilde{\phi}$ and $\phi_Y$ are surjective endomorphisms.
\end{proposition}

The next theorem is from the ``$\dim Y=2$" case in the proof of \cite[Theorem 1.2]{M23}.

\begin{theorem}\label{thm 1}
Let $X$ be a $\Q$-factorial normal projective variety admitting a polarized endomorphism $f$ and a fibration $\pi:X\rightarrow Y$ of relative dimension $1$. Suppose $Y$ is $\Q$-abelian. 
%and $g$ is an endomorphism on $Y$ such that $\pi\circ f=g\circ\pi$. 
Then $X$ is of log Calabi-Yau type.
\end{theorem}
\begin{proof}
Applying Proposition \ref{prop:T_f}, let $T_f$ be the finite sum of $f^{-1}$ periodic prime divisors on $X$, then $-(K_X+T_f)$ is $\Q$-movable. Thus we can find an effective $\Q$-divisor $D\sim_{\Q}-(K_X+T_f)$ which has no common components with $T_f$ and has coefficients not greater than 1. We claim that the pair $(X,T_f+D)$ is lc (thus log Calabi-Yau).

Since $X$ is smooth in codimension 1 and $\text{dim}(X/Y)=1$, general fiber of $X\rightarrow Y$ is smooth. Since the coefficients of the divisor $T_f+D$ are not greater than 1, the pair $(F,(T_f+D)|_F)$ is lc on a general fiber $F$. Then we have canonical bundle formula 
\[
0 \sim_{\Q} K_X+T_f+D\sim_{\Q}\pi^*(K_Y+B+M),
\]
where $(Y,B+M)$ is a generalized pair by Theorem \ref{thm:M b-nef}. Thus, $K_Y+B+M\sim_{\Q}0$. Since $K_Y\sim_{\Q} 0$, $B \ge 0$ and $M$ is pseudo-effective, we must have $B=0$, and $M\sim_{\Q}0$. 
Since $Y$ is $\Q$-abelian, it is 
%$\Q$-factorial 
klt and hence lc
(see \cite[Proposition 5.20]{KM98}).
By Lemma \ref{lemma:moduli part M=0}, $(Y,B+M)=(Y,M)$ is generalized lc. Hence $(X,T_f+D)$ is lc by Theorem \ref{thm:M b-nef}. This proves the claim and also the theorem.
%Since $g: Y \to Y$ is also polarized (see Proposition \ref{prop:polarized basic property}), it follows from Proposition \ref{prop:T_f} (2) $Y$ is lc, which, by Lemma \ref{lemma:moduli part M=0}, is equivalent to $(Y,B+M)=(Y,M)$ being generalized lc. Hence $(X,T_f+D)$ is lc by Theorem \ref{thm:M b-nef}. This proves the claim and also the theorem.
\end{proof}

\begin{lemma}\label{fiber P^2}
Let $X$ be a Gorenstein terminal 3-fold, $\pi:X\rightarrow Y$ a fibration to an elliptic curve $Y$, and $f:X\rightarrow X$ and $g:Y\rightarrow Y$ q-polarized endomorphisms satisfying $\pi\circ f=g\circ \pi$. Suppose $Y_1=\{y\in Y\, | \, X_y\cong \PP^2\}$ is Zariski dense in $Y$, where $X_y$ is the fiber of $\pi$ over $y$. Then $Y_1=Y$.
\end{lemma}

\begin{proof}
%proof of X_y reduced
By \cite[Lemma 5.2]{MZ18}, $X_y$ is irreducible for each $y\in Y$.

\textbf{Step 1.} Each fiber $X_y$ is reduced. 
Since terminal threefold $X$ has $\Sing X$ being a finite set \cite[Corollary 5.18]{KM98}, general fibres of $\pi$ are smooth (hence reduced). 
Take large $s$ such that 
$\deg g^s = q^s >\#\{y|X_y \text{ singular}\}$. Then some fibre $X_{y'}$ with $y' \in g^{-s}(y)$ is smooth so reduced. Thus $X_y$ is reduced too, noting that $g$ is \'etale and
\[(*) \hskip 1pc f^{s*}X_y=f^{s*}\pi^*y=\pi^*g^{s*}y.\]
In particular, $K_{X_y}^2=9$ for any $y \in Y$.
The $(*)$ also shows the branch locus $B_f$ and hence the ramification divisor $R_f$ contain no any fibre $X_y$.

\textbf{Step 2.}
A general fiber $X_y\cong\PP^2$.
Let
\[U=\{y\in Y \, | \, X_y \text{ is smooth}, -K_{X_y}=-K_X|_{X_y} \text{ is ample}\}.\]
Then $U$ intersects the dense set
$Y_1$, and hence a non-empty open subset of $Y$.
%so $Y_1\bigcap U\neq \varnothing$. Choose a point $y_0\in Y_1\bigcap U$. 
For any point $y\in U$, $X_y$ is a smooth Del Pezzo surface with $K_{X_y}^2 = 9$.
%=K_{X_{y_0}}^2=9$, 
By the classification, 
%of smooth Del Pezzo surfaces, 
$X_y\cong \PP^2$ as claimed.

\textbf{Step 3.}
Each fiber $X_y\cong \PP^2$.
Suppose the contrary that there exists some $y\in Y$ such that $X_y\ncong\PP^2$.  
Since a general fiber of $\pi$ is isomorphism to $\PP^2$ as showed, $f^{-s}(X_y)$ must contain a smooth fiber $F_1\cong \PP^2$ for large $s$. Then:

\begin{enumerate}
\item 
either there is some $t>0$ such that $F_2 := X_{g^t(y)}\cong \PP^2$ exists, or
\item 
for any $t\geq 0$, $X_{g^t(y)}\ncong \PP^2$.   
\end{enumerate}

For Case (1), we obtain $\varphi=f^t|_{X_y}:X_{y}\rightarrow F_2\cong\PP^2$, and $\psi=f^s|_{F_1}:F_1\cong \PP^2\rightarrow X_y$. Restricting the ramification divisor formula of $X$ to fibers, we get 
\[K_{X_y}=\varphi^*K_{F_2}+R_{f^t}|_{X_y}, \,\,\,
K_{F_1}=\psi^*K_{X_y}+R_{f^s}|_{F_1}.\]
Identify $F_1\cong F_2$, we get $h=\psi\circ \varphi:X_y\rightarrow X_y$, and 
\begin{align*}
K_{X_y}={}&\varphi^*(\psi^*K_{X_y}+R_{f^s}|_{F_1})+R_{f^t}|_{X_y}\\
       ={}&h^*K_{X_y}+\varphi^*(R_{f^s}|_{F_1})+R_{f^t}|_{X_y}
\end{align*}

For Case (2), $G_y :=\{g^r(y) \, | \, r\geq 0\}$ is a finite set. Hence some $y_1\in G_y$, is $g$-periodic %(say $g$-fixed after iteration of $f$) 
and $X_{y_1}\ncong \PP^2$. Suppose $g^t(y_1)=y_1$ for some $t>0$.  Then we have
\[K_{X_{y_1}}=h^*K_{X_{y_1}}+R_{f^t}|_{X_{y_1}},\]
where $h=f^t|_{X_{y_1}}$.
For both Cases (1) and (2), we find some fiber $F\ncong \PP^2$, with a non-isomorphic endomorphism $h$ on $F$ of $\deg$ $q^{2r}$ for some positive integer $r$, and 
\[K_F=h^*K_F+E\]
for some effective divisor $E$.

If $F$ is normal, then $F$ is a Gorenstein, Del Pezzo surface, with Picard number $\rho(F)=1$, since $F$ is dominanted by some fiber $F_1\cong\PP^2$. Also, $K_{F}^2=9$. Then $F\cong \PP^2$ by \cite[Lemma 8]{MiZ88}, contradicting our assumption.

Suppose $F$ is non-normal with $\nu:\widetilde{F}\rightarrow F$ the normalization. Then $\nu^*K_F=K_{\widetilde{F}}+C$, with $C$ the conductor (see \cite[\S 0.2]{R94}).
Since $\widetilde{F}$ is dominated by some $F_1 \cong \PP^2$, $\rho (\widetilde{F})=1$. Now $h$ lifts to $\widetilde{h}$ on $\widetilde{F}$, and $K_{\widetilde{F}}+C=\widetilde{h}^*(K_{\widetilde{F}}+C)+\nu^*E$. Then (all coefficients of) $C\leq 1$; see (\cite[Lemma 5.3]{NZ10}, arXiv ver) or Theorem \ref{Delta leq 1} late on. So $\rho(\widetilde{F})=1$, $(K_{\widetilde{F}}+C)^2=9$, and $C\leq 1$, which cannot happen by Reid's classification of non-normal Del Pezzo surface \cite[Theorem 1.1]{R94}. This finishes \textbf{Step 3} and also 
the proof of the theorem.
%Hence $\pi:X\rightarrow Y$ cannot have a fiber not isomorphic to $\PP^2$. This proves the theorem.
\end{proof}

\begin{proof1}[Proof of Theorem \ref{thm: lcy for Gorenstein 3-fold}]
Taking a small $\Q$-factorization as in Theorem \ref{thm:Q-fac-1} and by Lemma \ref{lemma:quasi etale or birational lcy}, we may assume that $X$ is already $\Q$-factorial (and Gorenstein terminal). Applying Proposition \ref{prop:etale cover of f-pair}, with $\phi=f$, there exists a quasi-\'etale finite cover $\mu:\widetilde{X} \to X$ from a normal projective variety and a Fano type fibration $\widetilde{X}\to Y$, where $Y$ is an Abelian variety. 
Further, $f$ lifts to $\widetilde{f}:\widetilde{X}\rightarrow\widetilde{X}$
and descends to $g:Y\rightarrow Y$ after iteration of $f$. 
%the following morphisms and pairs 
%\begin{center}
%\begin{tikzcd}
%(X,0)
%&(\widetilde{X},\widetilde{\Delta})
%\arrow[l,"\mu"']
%\arrow[r,"\pi"]
%&Y,
%\end{tikzcd}
%\end{center}
%exist  and satisfy 

%$\bullet \quad$ $\mu:(\widetilde{X},\widetilde{\Delta})\rightarrow (X,0)$ is a quasi-\'etale cover, i.e., $K_{\widetilde{X}}+\widetilde{\Delta}=\mu^*K_X$, 

%$\bullet \quad$  $\pi$ is a fiber space,

%$\bullet \quad$ $Y$ is an abelian variety,

%$\bullet \quad$ $(\widetilde{X},\widetilde{\Delta})$ is of Fano type over $Y$.

Applying Proposition   \ref{prop:polarized basic property}, 
$\widetilde{f}$ and $g$ are also $q$-polarized.
Since $\mu:\widetilde{X}\rightarrow X$ is quasi-\'etale, 
 $\widetilde{X}$ is Gorenstein, terminal singularity by \cite[Proposition 5.20]{KM98}.
By Theorem \ref{thm:Q-fac-1}, taking small $\Q$-factorization $\widetilde{X}'\to \widetilde{X}$ and further iteration of $\widetilde{f}$, $\widetilde{f}$ lifts to $\widetilde{f}'$ on $\widetilde{X}'$, and $\widetilde{X}'$ is Gorenstein terminal, of Fano type over $Y$ (see Lemma \ref{lemma:quasi etale or birational lcy}).  
Also by Lemma \ref{lemma:quasi etale or birational lcy}, $X$ is of Calabi-Yau type if and only if so is $\widetilde{X}'$. Thus we may replace $X,f$ with $\widetilde{X}',\widetilde{f}'$ to assume $X$ admits an $f$-equivariant Fano type fibration $\pi:X\rightarrow Y$, with $g$ on $Y$.
And from the construction of $Y$ in \cite[Theorem 6.5]{Y21}, $f^*K_X\equiv qK_X$ if $g^*=q\,\text{id}$ on $\N^1(Y)$.

Case 1: $\text{dim}Y=3$. Then $X=Y$ is abelian, and $(X,0)$ is log Calabi-Yau. 

Case 2: $\dim Y=0$. Then $(X,\Delta)$ is of Fano type (thus Calabi-Yau type).

Case 3: $\dim Y=2$. Then $\dim (X/Y)=1$, and the theorem follows from Theorem \ref{thm 1}.

Case 4: $\dim Y=1$. Then $g^*=q\id$ on $\N^1(Y)$, so $f^*K_X\equiv qK_X$. In particular, $K_X+\frac{R_f}{q-1}\equiv 0$.
Since $X$ is terminal, the general fiber of $\pi$ is smooth. By \cite[Propositions 3.4 and 4.1]{M23}, if a general fiber of  $\pi$ is not isomorphic to $\mathbb{P}^2$, the pair $(X, \frac{R_f}{q-1})$ is log Calabi-Yau after iteration of $f$. So we may assume $Y_1=\{y\in Y| X_y\cong \mathbb{P}^2\}$ is dense in $Y$. %where $X_y$ is the fiber of $\pi$ over $y$, 
%noting that the projective space is invariant under smooth projective deformation.

Then it follows from Lemma \ref{fiber P^2} that each fiber of $\pi$ is isomorphic to $\PP^2$. In particular, each fiber of $\pi$ is smooth. Then $\pi$ is smooth since $\pi$ is flat, hence $X$ is smooth; see \cite[Chapter III, Proposition 9.7, Theorem 10.2, Proposition 10.1]{H77}. Thus $X$ is of Calabi-Yau type by \cite[Theorem 1.2]{M23}.   
\end{proof1}

\section{Singularity of $f$-pairs;
%$(X,\Delta+\frac{R_{\Delta}}{q-1})$;
Proofs of Theorems \ref{thm:Delta+R/q-1 leq 1} - \ref{thm:(X,Delta) lc}, Corollary \ref{cor:quasi-'etale lcy}}\label{sec:singularity of (X,Delta+R/q-1)}
We begin with the following which includes Theorem \ref {thm:Delta+R/q-1 leq 1}.
\begin{theorem}\label{Delta leq 1}
Let $f: X \to X$ be a surjective endomorphism of a normal projective variety $X$, and $\Delta$ an effective Weil $\Q$-divisor such that $(X,\Delta)$ is an $f$-pair (see Definition \ref{def:f-pair}). Let $P$ be a prime divisor on $X$. Then:
\begin{enumerate}
\item[(1)] Either $\mu_P \Delta \leq 1$, or $P$ is $f^{-1}$-periodic with $f^{s*}P = P$ for some $s > 0$. Moreover, if $\mu_P \Delta \geq 1$, then $P$ is $f^{-1}$-periodic.

\item[(2)] Suppose $f$ is $q$-polarized. Then, after iteration of $f$, we have:
\[ {\text{\rm (all the coefficients of)}} \,\, \Delta + \frac{R_{\Delta}}{q-1} \leq 1 \]
Further, the equality $\mu_P(\Delta + \frac{R_{\Delta}}{q-1}) = 1$ holds  if and only if $P$ is $f^{-1}$-periodic.
In particular, $\Delta\leq 1$.
\end{enumerate}
    
\end{theorem}

\begin{proof}
(1) For any prime divisor $P$ on $X$, let $r_P$ denote the ramification index of $f$ at $P$. Define the exceptional set:
\[ S = \{ P \mid P \text{ is a prime divisor on } X,\ \mu_P \Delta \geq 1 \}. \]
For any $P$ in $S$ and prime divisor $Q$ with $f(Q) = P$, the non-negativity of $R_{\Delta} = \Delta + R_f - f^*\Delta$ implies
\[ (*) \hskip 1pc 0 \le \mu_Q R_{\Delta} = \mu_Q \Delta + r_Q - 1 - r_Q \mu_P \Delta. \]
This forces $\mu_Q \Delta \geq r_Q (\mu_P \Delta -1)+1\geq 1$, so $Q \in S$. Thus the finite closed set $T = \bigcup_{P \in S} P$ satisfies $f^{-1}(T) \subseteq T$, which implies $T = f^{-1}(T)$, so there is a positive integer $s$ such that $f^{-s}(P) = P$ for all $P \in S$. Set $R_{\Delta,s} := \sum_{k=0}^{s-1} (f^*)^k R_{\Delta}$, which satisfies
\[ K_X + \Delta = (f^s)^*(K_X + \Delta ) + R_{\Delta,s}. \]

For any $P \in S$ with $f^{s*}P = r_{P,s}P$, as in $(*)$ above, we compute its coefficient in $R_{\Delta,s}$:
\[ 0 \le \mu_P R_{\Delta,s} = (1 -r_{P,s})(\mu_P \Delta - 1). \]
Since $r_{P,s}\geq 1$ and $\mu_P \Delta\geq 1$, this implies $\mu_P \Delta = 1$ or $r_{P,s} = 1$. Hence $\mu_P \Delta \leq 1$, unless $P\in S$ and $f^{s*}P=P$. Also, if $\mu_P\Delta \geq 1$,  $P\in S$ and $P$ is $f^{-1}$-periodic.

(2)
We replace $f$ by its iteration and apply Theorem \ref{prop:Rf leq q-1} and Proposition \ref{prop:T_f}.
We may assume
$\frac{R_f}{q-1} \leq 1$;  
$\mu_P\frac{R_f}{q-1} = 1$ holds if and only if $P$ is $f^{-1}$-periodic, and indeed each of finitely many such  $P$ is $f^{-1}$-invariant.

\textit{Case 1: $P$ is $f^{-1}$-periodic,} or $f^{-1}$-invariant after iteration. Then
\[
f^*P = qP \implies \mu_PR_{\Delta} = (q-1)(1 - \mu_P \Delta ).
\]
Consequently:
\[
\mu_P\left(\Delta + \frac{R_{\Delta}}{q-1}\right) = \mu_P \Delta + \frac{1}{q-1}(q-1)(1 - \mu_P \Delta ) = 1.
\]

\textit{Case 2: $P$ is not $f^{-1}$-periodic.} By (1), $\mu_P \Delta < 1$. For any $k > 0$, $f^k(P)$ remains not $f^{-1}$-periodic, hence $r_{f^k(P)} \leq q-1$. Compute coefficient of $P$ in $\frac{R_{\Delta}}{q-1}$:
\begin{align*}
\mu_P\frac{R_{\Delta}}{q-1}
&= \frac{1}{q-1}\left(\mu_P \Delta + r_P - 1 - r_P\mu_{f(P)}\Delta \right) \\
&= \frac{1}{q-1}\left[(\mu_P \Delta - 1) + r_P(1 - \mu_{f(P)}\Delta )\right]\\
&< \frac{r_P}{q-1}.
\end{align*}
%Define 
%\[
%\alpha := \max\{\mu_T \Delta \mid T \text{ prime, non-}f^{-1}\text{-periodic}\} < 1, \quad 
%\beta := \min\{\mu_T \Delta \mid T \text{ prime}\} .
%\]
%This gives:
%\[
%\mu_P \Delta - 1 \leq \alpha - 1 < 0 \quad \text{and} \quad 1 - \mu_{f(P)}\Delta \leq 1 - \beta.
%\]
Replace $f$ by $f^s$, $R_{\Delta}$ by $R_{\Delta,s} = \sum_{k=0}^{s-1}(f^*)^kR_{\Delta}$, $q$ by $q^s$, and $r_P$ by:
\[
r_{P,s} = \prod_{k=0}^{s-1} r_{f^k(P)} \leq (q-1)^s,
\]
Then:
\[
\mu_P\frac{R_{\Delta,s}}{q^s-1}\leq \frac{r_{P,s}}{q^s-1}\leq \frac{(q-1)^s}{q^s-1}\to 0, \text{ when } s\to \infty.
\]
%\[
%\mu_P\frac{R_{\Delta,s}}{q^s-1}\leq \frac{(\alpha - 1) + (q-1)^s(1 - \beta)}{q^s - 1}\to 0, \text{ when } s\to \infty.
%\]

The inequality implies $\mu_P\frac{R_{\Delta,s}}{q^s-1}\to 0$ uniformly for all such $P$. Set $\alpha=\max\{\mu_T \Delta \mid T \text{ prime, }\mu_T\Delta<1\}$. Then, for all such $P$, $\mu_P\Delta\leq \alpha<1$. Thus, for some sufficiently large $s$, we have $\mu_P(\Delta + \frac{R_{\Delta,s}}{q^s-1}) <1$ for all such $P$.
\end{proof}

\begin{remark}
\begin{enumerate}
\item[(1)] Suppose $f^*P = P$. Replacing $\Delta$ by $\Delta' = \Delta + kP$ for some $k\in\Q$ such that $\Delta'$ is effective, $(X,\Delta')$ remains an $f$-pair with $R_{\Delta'}=R_{\Delta}$.
\item[(2)] From the proof of Theorem \ref{Delta leq 1}, $\mu_P\frac{R_{\Delta,s}}{q^s-1}\to 0$ uniformly for all non $f^{-1}$-periodic prime divisors.
\end{enumerate}
\end{remark}

\begin{lemma}[cf.~{\cite[Lemma 2.10]{BH14}}]\label{Nlc} 
Let $f: X \to X$ be a surjective endomorphism of a normal projective variety, and $\Delta$ an effective Weil $\Q$-divisor such that $K_X+\Delta$ is $\Q$-Cartier and $(X,\Delta)$ is an $f$-pair. Then each component of $\Nlc(X, \Delta)$ (resp. $\Nklt(X,\Delta)$) is $f^{-1}$-periodic.
%$\text{LC}(X,\Delta)$
\end{lemma}
\begin{proof}
The argument is exactly the same as that in \cite[Lemma 2.10]{BH14}.
\end{proof}

\begin{definition}[log canonical model]
Let \((X, \Delta)\) be a pair. A {\it log-canonical model} (or {\it lc model} for short) over the pair \((X, \Delta)\) is a proper birational morphism 
$\mu : Y \to X$
such that if we set
\[
\Delta_Y := \mu^{-1}_*(\Delta) + E^{\mathrm{lc}}_\mu,
\]
where \( E_{\mu}^{lc} \) is the sum of all the \(\mu\)-exceptional prime divisors taken with coefficient one, the pair \((Y, \Delta_Y)\) is log-canonical and \( K_Y + \Delta_Y \) is \(\mu\)-ample.
\end{definition}

(1) If there exists a log-canonical model over the pair \((X, \Delta)\), it is unique up to isomorphism \cite[Proposition 2.1]{OX12}.

(2) Suppose now that \(0\leq \Delta\leq 1\) and \(K_X + \Delta\) is \(\mathbb{Q}\)-Cartier. Then there exists a log-canonical model over \((X, \Delta)\) \cite[Theorem 1.1]{OX12}. Moreover the \(\mu\)-exceptional locus has pure codimension one \cite[Lemma 2.2]{OX12}. If we write
\[
  K_Y + \Delta_Y = \mu^*(K_X + \Delta) + \Delta_Y^{>1},\]
then \(\Delta_Y^{>1}\) is antieffective and \(\Supp\, \Delta_Y^{>1} = \text{Exc} (\mu)\), the exceptional locus of $\mu$ \cite[Lemma 2.4]{OX12}. Note that the divisor \(\Delta_Y^{>1}\) is \(\mathbb{Q}\)-Cartier,
since \(K_X + \Delta\) and \(K_Y + \Delta_Y\) are \(\mathbb{Q}\)-Cartier.

\begin{lemma}[cf.~{\cite[Lemma 2.11]{BH14}}]\label{lc model}
Let $f: X_1 \to X_2$ be a finite morphism of normal varieties with effective Weil $\mathbb{Q}$-divisors $\Delta_1,\Delta_2$ satisfying:
\[ K_{X_1} + \Delta_1 = f^*(K_{X_2} + \Delta_2). \]
If $\mu_2:(Y_2, \Delta_{Y_2})\to (X_2, \Delta_2)$ is an lc model over $(X_2, \Delta_2)$, then there is an lc model $(Y_1, \Delta_{Y_1})$ over $(X_1, \Delta_1)$, and \( f \) lifts to a finite morphism \( g : Y_1 \to Y_2 \), such that  $\mu_2 \circ g = f \circ \mu_1$, and
\[
K_{Y_1} + \Delta_{Y_1} = g^*(K_{Y_2} + \Delta_{Y_2}) .
\]
\end{lemma}
\begin{proof}
Let $Y_1$ be the normalization of $X_1\times _{X_2}Y_2$. Denote by $\mu_1: Y_1\to X_1$ and $g: Y_1\to Y_2$ the projections. Let $E_{\mu_i}$ be the sum of $\mu_i$-exceptional prime divisors. By the definition of lc model, $\Delta_{Y_2}=\mu_{2*}^{-1}\Delta_2+E_{\mu_2}$.
Let $\Delta_{Y_1}=\mu_{1*}^{-1}\Delta_1+E_{\mu_1}$. 
From the ramification divisor formula, 
$f^*\Delta_2-\Delta_1=R_f$.
We derive, by taking strict transforms on $Y_1$ and $Y_2$:
\[
g^*\mu_{2*}^{-1}\Delta_2 - \mu_1^{-1}\Delta_1 = R_g^h,
\]
where we decompose $R_g = R_g^h + R_g^v$ into $\mu_1$-non-exceptional and $\mu_1$-exceptional components.
Combining this with $K_{Y_1} = g^*K_{Y_2} + R_g$,
we obtain
\[
K_{Y_1}+ \mu_1^{-1}\Delta_1= g^*(K_{Y_2} + \mu_{2*}^{-1}\Delta_2)+R_g^v.
\]
Since 
%$g^{-1}E_{\mu_2} = E_{\mu_1}$,
$R_g^v=g^*E_{\mu_2}-E_{\mu_1}$,
\[
K_{Y_1} + \Delta_{Y_1} = g^*(K_{Y_2} + \Delta_{Y_2}).
\]

%[K_{Y_1}+\mu_{1*}^{-1}\Delta_1+E_{\mu_1}=g^*(K_{Y_2}+\mu_{2*}^{-1}\Delta_2+E_{\mu_2})+G,\]
%\[K_{Y_1}+\Delta_{Y_1} =g^*(K_{Y_2}+\Delta_{Y_2})+G\]
%where $G$ is $\mu_1$-exceptional, having no common component with %$E_{\mu_1}$. Thus $G=0$. 
%Let $\Delta_{Y_1}=\mu_{1*}^{-1}\Delta_1+E_{\mu_1}$, 
Then $(Y_j, \Delta_{Y_j})$ is lc and $K_{Y_j}+\Delta_{Y_j}$ is $\mu_{j}$-ample, for $j =2 $ and hence for $j =1$ (cf. \cite[Proposition 5.20]{KM98}). So $(Y_1,\Delta_{Y_1})$ is an lc model over $(X_1,\Delta_1)$.
\end{proof}

\begin{proposition}[cf.~{\cite[Proposition 3.1]{BH14}}]\label{Nlc R_f}
   Let $f: X \to X$ be a surjective endomorphism on a normal projective variety $X$ and $\Delta \le 1$ an effective Weil $\Q$-divisor on $X$ such that $K_X+\Delta$ is $\Q$-Cartier and $(X,\Delta)$ is an $f$-pair. 
   Let $Z$ be an irreducible component of $\Nlc(X, \Delta)$ (which is $f^{-1}$-periodic by Lemma~\ref{Nlc}). Suppose that $f^{-1}(Z) = Z$. Then 
    \[
        Z \nsubseteq \Supp R_{\Delta}.
    \]
\end{proposition}

\begin{proof}
\textbf{Step 0.}
    Since $\Delta \le 1$, $\Codim_X Z \geq 2$.
    Let $\mu: Y \to (X, \Delta)$ be a log resolution of $(X, \Delta)$.  Write
\[
     K_Y +\mu_*^{-1}\Delta+ \sum a_iE_i+\sum
     b_jE_j'= \mu^*(K_X + \Delta), 
    \]
       where $E_i$ runs over $\mu$-exceptional prime divisors satisfying $Z\subseteq\mu(E_i)$ and $a_i>1$, while each $E_j'$ is a $\mu$-exceptional prime divisor such that $Z\nsubseteq \mu(E_j')$ or $b_j\leq 1$. Then each $E_i$ is a non-lc place of $(X,\Delta)$, and $\mu(E_i)=Z$ since $Z$ is an irreducible component of $\Nlc(X)$.
    
    \textbf{Step 1.} Suppose $Z \subseteq \Supp R_\Delta$ on the contrary. Since $K_X+\Delta$ is $\Q$-Cartier, $R_\Delta$ is $\mathbb{Q}$-Cartier with Cartier index $m \ge 1$. For $\ell > 0$,
    \[
        K_X + \Delta = (f^{\ell})^*(K_X + \Delta) + R_{\Delta,\ell}, \quad \text{where} \quad R_{\Delta,\ell} = \sum_{j=0}^{\ell-1} (f^j)^* R_{\Delta}.
    \]
    As $Z \subseteq \Supp R_{\Delta}$ and $f^{-1}(Z) = Z$, we have
    \[
        \mult_{E_i} \mu^*(f^j)^* R_{\Delta} \geq \frac{1}{m} 
    \]
    for any $j\geq 0$.
    Thus for all positive integers $i$ and sufficiently large $\ell$,
    \[
         a_i-\mult_{E_i} \mu^* R_{\Delta, \ell}\leq 1 .
    \]
    So we may assume $a_i-\mult_{E_i}\mu^*R_{\Delta}\leq 1$ for all $i$, after iteration of $f$. 
    
    \textbf{Step 2.}
     Set $c_i=\mult_{E_i} \mu^* R_{\Delta}$, $d_j=\mult_{E_j'} \mu^* R_{\Delta}$. We can write
    \[
     K_Y +\mu_*^{-1}\Delta- \mu_*^{-1}R_{\Delta}+ \sum(a_i-c_i)E_i+\sum(b_j-d_j)E_j'= \mu^*(K_X + \Delta-R_{\Delta}), 
    \]
   By Step 2, $a_i-c_i\leq 1$. 
   Also, if $Z\subseteq \mu(E_j')$, then $b_j\leq 1$ by the assumption, so $b_j-d_j\leq 1$. Thus $(X,\Delta-R_{\Delta})$ is sub-lc over generic point of $Z$, which leads to a contradiction by \cite[Proposition 5.20]{KM98}, since $K_X + \Delta-R_{\Delta}=f^*(K_X + \Delta)$, $f^{-1}(Z)=Z$ and $Z$ is a non-lc center of $(X,\Delta)$. This proves the proposition.
\end{proof}

\begin{remark}\label{rmk: proof of num lc}
    (1) Proposition \ref{Nlc R_f} generalizes the result of \cite[Proposition 3.1]{BH14} and simplifies the proof there. Also,  we do not assume $\mu:Y\to X$ is the log canonical model over $X$ as in the proof of \cite[Proposition 3.1]{BH14}, but only a log resolution of $(X,\Delta)$.
    
    (2) When $X$ is  a surface, not assuming $K_X+\Delta$ to be $\Q$-Cartier, replace $\Nlc(X,\Delta)$ by $\Nlc_{num}(X,\Delta)$: the numerically non-lc locus of $(X,\Delta)$ (see \cite[section 20]{CMZ20}), which is a finite set. Suppose $Z\subseteq \Supp R_{\Delta}$. Then a similar proof as in Proposition \ref{Nlc R_f} shows that $K_X+\Delta-R_{\Delta}$ is numerically sub-lc over $Z$, after iteration of $f$, which leads to a contradiction as in Step 2 of the proof.
\end{remark}

\begin{theorem}[cf.~{\cite[Theorem 1.4]{BH14}}]\label{Nlc deg}
    In the setting of Proposition \ref{Nlc R_f}, $\deg(f|_Z) = \deg f$.
\end{theorem}

\begin{proof}
    It follows from Theorem \ref{Delta leq 1} that $\Delta\leq 1$. Then the argument is exactly the same as in \cite[Theorem 1.4]{BH14}, but we apply Propositions \ref{lc model} and \ref{Nlc R_f}, instead.
\end{proof}

\begin{proof1}[Proof of Theorem \ref{thm:(X,Delta) lc}]
    Assume to the contrary that $(X, \Delta)$ is not log canonical. Let $Z$ be an irreducible component of $\Nlc(X, \Delta)$. By Lemma \ref{Nlc}, after replacing $f$ with a suitable iteration, we may assume $f^{-1}(Z) = Z$. Applying Theorem \ref{Nlc deg}, we obtain $\deg (f|_Z)=\deg f$, which contradicts \cite[Lemma 3.11]{M20}.
\end{proof1}

\begin{proof1}[Proof of Corollary \ref{cor:quasi-'etale lcy}]
By the ramification divisor formula and the assumption,
$K_X + \Delta = f^*(K_X + \Delta) + R_{\Delta} = f^*(K_X + \Delta)$. Thus $K_X+\Delta\equiv 0$ because every eigenvalue of $f^*$ on $\N^1(X)$ has the modulus $>1$ by \cite[Theorem 1.1]{MZ20}. Then $K_X+\Delta \sim_{\Q}0$ by Gongyo \cite[Theorem 1.2]{G13}.
\end{proof1}

\section{Ramification formula under subadjunction; Proof of Theorem \ref{thm:polarized-inversion-of-adjunction}}\label{sec:ramification-subadj}

Throughout this section, {\it $f:X\rightarrow Y$ (or $f: X\rightarrow X$) is a finite surjective morphism between normal projective varieties}. We study the behavior of ramification formula under subadjunction and prove Theorem \ref{thm:polarized-inversion-of-adjunction}.

The following is the main result of this section.

\begin{theorem}\label{thm:subadj-ramification}
Let $f: X \to X$ be a polarized endomorphism of a normal projective variety and $\Delta$ an effective $\Q$-divisor such that $K_X+\Delta$ is $\Q$-Cartier, and $(X,\Delta)$ is an $f$-pair. Let $Z$ be an lc center of $(X,\Delta)$ satisfying $f^{-1}(Z)=Z$, and $(K_X+\Delta)|_{\widetilde{Z}}=K_{\widetilde{Z}}+\Delta_{\widetilde{Z}}+M_{\widetilde{Z}}$ the subadjunction fomula, where $\widetilde{Z}$ is the normalization of $Z$ (see Definition \ref{def:subadjunction}). Then:
\begin{enumerate}
    \item $K_{\widetilde{Z}}+\Delta_{\widetilde{Z}}=f|_{Z}^*(K_{\widetilde{Z}}+\Delta_{\widetilde{Z}})+R_{\Delta}|_{\widetilde{Z}}$; $f|_{\widetilde{Z}}$ is $q$-polarized.

    \item $f|_{\widetilde{Z}}^*M_{\widetilde{Z}}\sim_{\Q} M_{\widetilde{Z}}$, and $M_{\widetilde{Z}}\equiv_w 0$.
\end{enumerate}
\end{theorem}

\begin{remark}\label{rmk:Z nsubseteq R_Delta}
    The restriction $R_{\Delta}|_{\widetilde{Z}}$ is well defined since $R_{\Delta}$ is $\Q$-Cartier and since $Z\nsubseteq \Supp R_{\Delta}$ by applying \cite[Proposition 5.20]{KM98} to $K_X + \Delta - R_{\Delta} = f^*(K_X + \Delta)$.
\end{remark}

The following is clear from the subadjuction formula in Definition \ref{def:subadjunction} (4).

\begin{lemma}\label{codim1 adj}
Suppose that $K_X+P+\Delta_X=f^*(K_Y+Q+\Delta_Y)+R_{\Delta}$, where $R_{\Delta}\geq0$, $P$ and $Q$ are prime divisors on $X$ and $Y$, respectively, satisfying $f(P)=Q$, $P\nsubseteq \Supp \Delta_X$, $Q\nsubseteq\Supp \Delta_Y$, and $K_X+P+\Delta_X$ and $K_Y+Q+\Delta_Y$ are $\Q$-Cartier. Let $\widetilde{P}$ and $\widetilde{Q}$ be normalizations of $P$ and $Q$, respectively. Write $K_{\widetilde{P}}+\Delta_{\widetilde{P}}=(K_X+P+\Delta_X)|_{\widetilde{P}}$,  $K_{\widetilde{Q}}+\Delta_{\widetilde{Q}}=(K_X+Q+\Delta_X)|_{\widetilde{Q}}$ as in Definition \ref{def:subadjunction} (4).
Then $K_{\widetilde{P}}+\Delta_{\widetilde{P}}=f|_{\widetilde{P}}^*(K_{\widetilde{Q}}+\Delta_{\widetilde{Q}})+R_{\Delta}|_{\widetilde{P}}$.
\end{lemma}

\begin{corollary}\label{codim1 end}
Suppose that $K_X+P+\Delta=f^*(K_X+P+\Delta)+R_{\Delta}$, where $R_{\Delta}\geq0$, $P$ is a prime divisor on $X$ satisfying $f(P)=P$, $P\nsubseteq \Supp \Delta$, and $K_X+P+\Delta$ is $\Q$-Cartier.
Let $K_{\widetilde{P}}+\Delta_{\widetilde{P}}=(K_X+P+\Delta)|_{\widetilde{P}}$ be as in Definition \ref{def:subadjunction} (4), where $\widetilde{{P}}$ is the normalization of $P$. Then $K_{\widetilde{P}}+\Delta_{\widetilde{P}}=f|_{\widetilde{P}}^*(K_{\widetilde{P}}+\Delta_{\widetilde{P}})+R_{\Delta}|_{\widetilde{P}}$.
\end{corollary}

\begin{theorem}\label{fiber sp adj}
Suppose we have the commutative diagram of normal projective varieties:

\begin{center}
 \begin{tikzcd}
X\arrow[d,"\pi_X"]\arrow[r,"f"]
&Y\arrow[d,"\pi_Y"]\\
Z\arrow[r,"g"]
&W
\end{tikzcd}   
\end{center}
where $f$,  $g$ are finite surjective, and $\pi_X$, $\pi_Y$ are algebraic fibrations. 
Let $\Delta_X$ and $\Delta_Y$ be Weil $\Q$-divisors on $X$ and $Y$, respectively, such that $(X,\Delta_X)\to Z$ and $(Y,\Delta_Y)\to W$ satisfy conditions in Definition \ref{def:canonical bundle formula}, and $R_{\Delta,X}:=K_X+\Delta_X-f^*(K_Y+\Delta_Y)$ equals the pull back of some $\Q$-Cartier divisor $R_{\Delta,Z}$ via $\pi_X$. Let $K_X+\Delta_X\sim_{\Q}\pi_X^*(K_Z+\Delta_Z+M_Z)$, $K_Y+\Delta_Y\sim_{\Q}\pi_Y^*(K_W+\Delta_W+M_W)$ be the canonical bundle formula as in Definition \ref{def:canonical bundle formula}.

Then $K_Z+\Delta_Z=g^*(K_W+\Delta_W)+R_{\Delta,Z}$ and $M_Z\sim_{\Q}g^*M_W$.
\end{theorem}

\begin{proof}
Since $R_{\Delta,X}=\pi_X^*R_{\Delta,Z}$, $R_{\Delta,X}$ is $\Q$-Cartier and does not dominate $Z$. Thus the sub-pair $(X,\Delta_X-R_{\Delta ,X})$ is still sub-lc over generic point of $Z$. Then we have the canonical bundle formula (noting that adding a pullback to the boundary does not change the moduli part), $K_X+\Delta_X-R_{\Delta,X}\sim_{\Q}\pi_X^*(K_Z+\Delta_Z'+M_Z)$, where $\Delta_Z'=\Delta_Z-R_{\Delta,Z}$. By \cite[theorem 3.2]{A99}, $K_Z+\Delta_Z'=g^*(K_W+\Delta_W)$ (noting that the condition ``$\nu$ is the base change of $\sigma$" in \cite[theorem 3.2]{A99} can be replaced as ``$\nu,\sigma $ are finite surjective"), i.e.,
\[K_Z+\Delta_Z=g^*(K_W+\Delta_W)+R_{\Delta,Z}.\]
One the other hand, from the ramification divisor formula of $f:(X,\Delta_X)\to(Y,\Delta_Y)$,
%By the commutative diagram and calculating $K_X + \Delta$ in two different ways,
%$R_{\Delta, Z}$,
\[K_Z+\Delta_Z+M_Z\sim_{\Q}g^*(K_W+\Delta_W+M_W)+R_{\Delta,Z}.\]
Combining the above two displayed equalities, we get $M_Z\sim_{\Q}g^*M_W$.
\end{proof}

\begin{theorem}\label{codim>1 adj}
Suppose that $K_X+\Delta_X=f^*(K_Y+\Delta_Y)+R_{\Delta}$, where $R_{\Delta}\geq 0$, $\Delta_X$, $\Delta_Y$ are effective Weil $\Q$-divisors on $X$ and $Y$, respectively, such that $K_X+\Delta_X$, $K_Y+\Delta_Y$ are $\Q$-Cartier. Let $Z$, $W$ be lc centers of $(X,\Delta_X)$, $(Y,\Delta_Y)$ rspectively, such that $f(Z)=W$. Let $\widetilde{Z}$, $\widetilde{W}$ be the normalizations of $Z$, $W$. Apply the subadjunction formula $(K_X+\Delta_X)|_{\widetilde{Z}}\sim_{\Q} K_{\widetilde{Z}}+\Delta_{\widetilde{Z}}+M_{\widetilde{Z}}$, $(K_Y+\Delta_Y)|_{\widetilde{W}}\sim_{\Q} K_{\widetilde{W}}+\Delta_{\widetilde{W}}+M_{\widetilde{W}}$ as in Definition \ref{def:subadjunction}. 

Then $K_{\widetilde{Z}}+\Delta_{\widetilde{Z}}=f|_{\widetilde{Z}}^*(K_{\widetilde{W}}+\Delta_{\widetilde{W}})+R_{\Delta,{\widetilde{Z}}}$, and $M_{\widetilde{Z}}\sim_\Q f|_{\widetilde{Z}}^*M_{\widetilde{W}}$.
\end{theorem}
\begin{proof}
Let $\pi_Y:Y'\to Y$ be a log resolution of $(Y,\Delta_Y)$ such that $(Y,\Delta_Y)$ has an lc place $Q$ on $Y'$, with $\pi_Y(Q)=W$. Let $X'$ be the normalization of $X\times_Y Y'$, and $\pi_X: X'\to X $ and $g: X'\to Y'$ the projections . Let $P$ be a prime divisor on $X'$ such that $g(P)=Q$, $\pi_X(P)=Z$. Then $P$ is an lc place of $(X,\Delta_X-R_{\Delta})$ by \cite[Proposition 5.20]{KM98}. But $Z$ is already an lc center of $(X,\Delta_X)$, which implies $Z\nsubseteq \Supp R_{\Delta}$. Thus $P\nsubseteq\Supp \pi^*R_{\Delta}$, and $P$ is an lc place of $(X,\Delta_X)$. 
Write uniquely (with $\pi_{X*}\Delta_{X'}=\Delta_X$ and $\pi_{Y*}\Delta_{Y'}=\Delta_Y$):
\[K_{X'}+P+\Delta_{X'}=\pi_X^*(K_X+\Delta_X),\quad K_{Y'}+Q+\Delta_{Y'}=\pi_Y^*(K_Y+\Delta_Y).\]
Then $K_{X'}+P+\Delta_{X'}=g^*(K_{Y'}+Q+\Delta_{Y'})+\pi_X^*R_{\Delta}$.

By Lemma \ref{codim1 adj}, if $\widetilde{P}$ and $\widetilde{Q}$ are the normalizations of $P$ and $Q$, with subadjunction formula $K_{\widetilde{P}}+\Delta_{\widetilde{P}}=(K_{X'}+P+\Delta_{X'})|_{\widetilde{P}}$ and $K_Q+\Delta_Q=(K_{Y'}+Q+\Delta_{Y'})|_{\widetilde{Q}}$, then
\[K_{\widetilde{P}}+\Delta_{\widetilde{P}}=g|_{\widetilde{P}}^*(K_{\widetilde{Q}}+\Delta_{\widetilde{Q}})+(\pi^*_XR_{\Delta})|_{\widetilde{P}}.\]

Consider fibrations $\pi_{\widetilde{P}}:(\widetilde{P},\Delta_{\widetilde{P}})\to \widetilde{Z}$,  $\pi_{Q}:(Q,\Delta_Q)\to \widetilde{W}$, and note that $(\pi_X^*R_{\Delta})|_{\widetilde{P}}=\pi_{\widetilde{P}}^*(R_{\Delta}|_{\widetilde{P}})$. Applying Theorem \ref{fiber sp adj},
\[K_{\widetilde{Z}}+\Delta_{\widetilde{Z}}=f|_{\widetilde{Z}}^*(K_{\widetilde{W}}+\Delta_{\widetilde{W}})+R_{\Delta}|_{\widetilde{Z}},\]
and $M_Z\sim_{\Q}f|_{\widetilde{Z}}^*M_{\widetilde{W}}$.
\end{proof}

\begin{proof1}[Proof of Theorem \ref{thm:subadj-ramification}]
Note that $f|_{\widetilde{Z}}$ is polarized by Proposition \ref{prop:polarized basic property}.
If $\Codim_X Z=1$, then the theorem follows from Corollary \ref{codim1 end} with $M_{\widetilde{Z}} = 0$.

If $\Codim_X Z>1$, then it follows from proposition \ref{codim>1 adj} that 
\[K_{\widetilde{Z}}+\Delta_{\widetilde{Z}}=f|_{\widetilde{Z}}^*(K_{\widetilde{Z}}+\Delta_{\widetilde{Z}})+R_{\Delta}|_{\widetilde{Z}},\] 
and $M_{\widetilde{Z}}\sim_{{\Q}}f|_{\widetilde{Z}}^*M_{\widetilde{Z}}$.
The moduli of eigenvalues of $f|_{\widetilde{Z}}^*$ on $\N_{n-1}(\widetilde{Z})$ are greater than $1$ by \cite[theorem 3.3]{M20}. Hence $M_{\widetilde{Z}}\equiv_w 0$. This proves the theorem.
\end{proof1}

\begin{proof1}[Proof of Theorem \ref{thm:polarized-inversion-of-adjunction}]
By Theorem \ref{thm:subadj-ramification}, $(1)$ follows, and also $M_{\widetilde{Z}} \equiv 0$ by assumption and 
\cite[Lemma 2.3]{MZ18}.
For (2), $Z$ is also an lc center of $(X, \Delta + \frac{R_{\Delta}}{q-1})$, since $Z\nsubseteq \Supp R_{\Delta}$ by Remark \ref{rmk:Z nsubseteq R_Delta}.
$f|_{\widetilde{Z}}$ is also $q$-polarized and $f|_{\widetilde{Z}}^*M_{\widetilde{Z}}\sim_{\Q} M_{\widetilde{Z}}$ by Theorem \ref{thm:subadj-ramification}, so
$M_{\widetilde{Z}}\sim_{\Q} 0$ 
by Proposition \cite[Proposition 3.3]{MZ22}. Then the pair $(K_{\widetilde{Z}},\Delta_{\widetilde{Z}}+\frac{R_{\Delta}|_{\widetilde{Z}}}{q-1})$ is lc if and only if $(K_{\widetilde{Z}},\Delta_{\widetilde{Z}}+\frac{R_{\Delta}|_{\widetilde{Z}}}{q-1}+M_{\widetilde{Z}})$ is generalized lc by Lemma \ref{lemma:moduli part M=0}, if and only if $(X,\Delta+\frac{R_{\Delta}}{q-1})$ is lc near $Z$ by Definition \ref{def:subadjunction} (3). Here we use the fact that the moduli part of the restriction to ${\widetilde Z}$ of the pairs $(X, \Delta)$ and 
   $(X, \Delta + \frac{R_{\Delta}}{q-1})$ are the same $M_{\widetilde Z}$ since adding a pullback (of the  restriction of $\frac{R_{\Delta}}{q-1}$ to $Z$) to the boundary does not change the moduli part.
\end{proof1}

\section{Reduction to $\rho(X) = 1$; Proofs of Theorem \ref{thm:klt to Picard 1} and Corollary \ref{cor:numerically trivial induction}} \label{sec:klt to picard 1}

In this section, we prove Theorem \ref{thm:klt to Picard 1}, i.e.,  to reduce klt case of Conjecture \ref{conj Delta} to Picard number 1 case. Our method is to use equivariant log MMP (cf.~Theorem \ref{thm:polarized-emmp}) to make induction. We begin with the following concerning %divisorial 
birational contraction.

\begin{proposition}\label{prop:birational-ramification-formula}
Let $\pi:X\to Y$ be a birational morphism between normal projective $\Q$-factorial varieties, and $f$ and $g$, $q$-polarized endomorphism on $X$ and $Y$, respectively, such that $\pi\circ f=g\circ \pi$. Let $\Delta_X$ be an effective Weil $\Q$-divisor such that $(X,\Delta_X)$ is an $f$-pair. Let $\Delta_Y=\pi_*\Delta_X$, $R_{\Delta_Y}=\pi_*R_{\Delta_X}$. Then:
\begin{enumerate}
    \item $K_Y+\Delta_Y=g^*(K_Y+\Delta_Y)+R_{\Delta_{Y}}$, so if $(X, \Delta)$ is an $f$-pair, then
    $(Y, \Delta_Y)$ is a $g$-pair.

    \item $K_X+\Delta_X+\frac{R_{\Delta_X}}{q-1}=\pi^*(K_Y+\Delta_Y+\frac{R_{\Delta_Y}}{q-1})$.

\end{enumerate}
\end{proposition}

\begin{proof}
(1)$K_X+\Delta_X=f^*(K_X+\Delta_X)+R_{\Delta_X}$ is equivalent to
$R_f=f^*\Delta_X+R_{\Delta_X}-\Delta_X$.
Taking $\pi_*$, we get
$R_g=g^*\Delta_Y+R_{\Delta_Y}-\Delta_Y$,
which is equivalent to 
$K_Y+\Delta_Y=g^*(K_Y+\Delta_Y)+R_{\Delta_{Y}}$.

(2) The proof is the same as \cite[Proposition 3.3]{M23}
with $K_X$ and $K_Y$ replaced by $K_X+\Delta_X$ and $K_Y+\Delta_Y$, respectively.
\end{proof}

For a Fano contraction, we obtain a log Calabi Yau structure from a general fiber and the base. First, we slightly generalize the definition and results in \cite{Y21}.
\begin{definition}[{\cite[Definition 3.6]{Y21}}] \label{def:mori-fiber-canonical}
    A morphism \(\pi\colon (X,\Delta_X) \to (Y,\Delta_Y)\) of pairs is called a \textit{Mori fiber space of canonical bundle formula type} if
    \begin{enumerate}
        \item \(X\) is a \(\mathbb{Q}\)-factorial normal projective variety and \(\Delta_X\) is an effective Weil \(\mathbb{Q}\)-divisor such that \((X,\Delta_X)\) is lc,
        
        \item \(\pi\colon X \to Y\) is a \((K_{X} + \Delta_X)\)-Mori fiber space, and
        
        \item for any prime divisor \(E\) on \(Y\), \(\Delta_Y\) satisfies
        \[
            \operatorname{ord}_{E}(\Delta_Y) = \frac{m_{E} - 1 + \operatorname{ord}_{F}(\Delta_X)}{m_{E}},
        \]
    \end{enumerate}
        where \(F\) is a prime divisor on \(X\) satisfying \(\pi^{*}E = m_{E} F\) for some integer \(m_{E} > 0\).
    
\end{definition}

\begin{proposition}[{\cite[Proposition 3.8]{Y21}}]\label{prop:mori fiber klt}
    In the setting of Definition \ref{def:mori-fiber-canonical}, if $(X,\Delta)$ is lc (resp. klt), the pair $(Y,\Delta_Y)$ is lc (resp. klt).
\end{proposition}
\begin{proof}
     The lc part follows from replacing $(Y,B)$ with $(Y,B+M)$ (see Theorem \ref{thm:M b-nef}) in the proof of \cite[Proposition 3.8]{Y21}. The klt part is just \cite[Proposition 3.8]{Y21}.
\end{proof}

\begin{proposition}[{\cite[Proposition 3.9]{Y21}}]\label{prop:mori-fiber-gpair}
Let $\pi:(X,\Delta_X)\to (Y,\Delta_Y)$ be a Mori fiber space of canonical bundle type, and $f$ and $g$ $q$-polarized endomorphisms on $X$ and $Y$, respectively, such that $\pi\circ f=g\circ \pi$. Suppose $(X,\Delta_X)$ is an $f$-pair. Then:

\begin{enumerate}
    \item $(Y,\Delta_Y)$ is a $g$-pair, i.e., $R_{\Delta_Y} \ge 0$.
    \item $R_{\Delta_X}-\pi^*R_{\Delta_Y}$ is an effective divisor and it has no vertical components of $\pi$.
\end{enumerate}
\end{proposition}

\begin{proposition}\label{prop:mori-fiber-lcy}
In the setting of Proposition \ref{prop:mori-fiber-gpair} with $\dim Y>0$,
suppose Conjecture \ref{conj Delta} holds for both $(Y,\Delta_Y)$ and $(F,\Delta_X|_F)$ with $F$ a general $f$-periodic fiber (which exists by Proposition \ref{prop:zd of f-periodic pt} and which is clearly an $f|_F$-pair when $F$ is $f$-invariant). Then Conjecture \ref{conj Delta} holds for $(X,\Delta_X)$ as well.
\end{proposition}

\begin{proof}
Define $\Delta_Y$ as in Definition \ref{def:mori-fiber-canonical}. Then $(Y,\Delta_Y)$ is klt by Proposition \ref{prop:mori fiber klt} and $R_{\Delta_Y}\geq 0$ by Proposition \ref{prop:mori-fiber-gpair}.
Choose an ample divisor $H$ such that $f^*H\sim qH$. Since $\dim \N^1 (X)/\pi^*\N^1(Y)=1$, $H$ is the relatively generator. Then $q(K_X+\Delta_X)-f^*(K_X+\Delta_X)\equiv_{\pi} 0$ (thus $\sim_{\pi,\Q} 0$ by the Cone Theorem \cite[Theorem 3.7]{KM98}).
Note that
\begin{align*}
(*) \hskip 1pc q(K_X+\Delta_X)-f^*(K_X+\Delta_X)&=(q-1)(K_X+\Delta_X)+R_{\Delta_X} \\
                               &=(q-1)(K_X+\Delta_X+\frac{R_{\Delta_X}}{q-1}),
\end{align*} 
%Take an ample divisor $H$ on $X$ such that $f^*H\sim qH$. Suppose
%\[K_X+\Delta_X=aH+\pi^*D,\]
%where $a\in \Q_{< 0}$, and $D$ is a $\Q$-Cartier divisor on $Y$. Then
%\begin{align*}
%q(K_X+\Delta_X)-f^*(K_X+\Delta)&= q(aH+\pi^*D)-f^*(aH+\pi^*D) \\
%                               &\sim_{\Q}q\pi^*D-f^*\pi^*D 
%                                \sim_{\pi,\Q}0,
%\end{align*}
so $K_X+\Delta_X+\frac{R_{\Delta_X}}{q-1}\sim_{\pi,\Q} 0$.

Take a general $f$-periodic fiber $F$. We may assume that $F$ is $f$-invariant after iteration of $f$. Then
\[K_F+\Delta_X|_F=(f|_F)^*(K_F+\Delta_F)+R_{\Delta_X}|_F.\]
By our assumption, $(F,\Delta_X|_F+\frac{R_{\Delta_X}|_F}{q-1})$ is lc after iteration of $f$, so $(X,\Delta_X+\frac{R_{\Delta_X}}{q-1})$ is lc over the generic point of $Y$. Then we can write the canonical bundle formula
\[(**) \hskip 1pc K_X+\Delta_X+\frac{R_{\Delta_X}}{q-1}\sim_{\Q}\pi^*(K_Y+B+M).\]

\begin{claim}\label{claim 6.6}
    $B\geq \Delta_Y+\frac{R_{\Delta_Y}}{q-1}$.
\end{claim}

Assuming the claim and by the assumption,
%since $\dim Y<n$,
\[K_Y+B+M=K_Y+\Delta_Y+\frac{R_{\Delta_Y}}{q-1}+B_1+M\sim_{\Q}B_1+M,\]
where $B_1 = B - (\Delta_Y+\frac{R_{\Delta_Y}}{q-1}) \geq 0$. We get 
\[(***) \hskip 1pc K_X+\Delta_X+\frac{R_{\Delta_X}}{q-1}\sim_{\Q}\pi^*(B_1+M).\]
On the other hand, letting $n = \dim X$, we calculate:
\[f^*(K_X+\Delta_X).H^{n-1}=\frac{1}{q^{n-1}}f^*(K_X+\Delta_X).(f^*H)^{n-1}=q(K_X+\Delta_X).H^{n-1}.\]
This and the $(*)$ above imply
\[(K_X+\Delta_X+\frac{R_{\Delta_X}}{q-1}).H^{n-1}
=\frac{1}{q-1}(q(K_X+\Delta_X)-f^*(K_X+\Delta_X)).H^{n-1}=0,\]
which, by the $(***)$ above, is equivalent to
\[\pi^*(B_1+M).H^{n-1}=0.\]
Since $B_1+M$ is pseudo-effective and $H$ is ample, we have $B_1+M\equiv 0$ by Lemma \ref{lemma:B.H^(n-1)=0}, which implies $B_1=0$, $M\equiv 0$ and $K_X+\Delta_X+\frac{R_{\Delta_{X}}}{q-1}\equiv 0$ by the $(***)$ above. Then $(Y,B+M)$ is generalized lc if and only if $(Y,\Delta_Y+\frac{R_{\Delta_Y}}{q-1})$ is lc (cf. Lemma \ref{lemma:moduli part M=0}), which holds by assumption.
Hence $(X,\frac{R_{\Delta_X}}{q-1}+\Delta_X)$ is lc (cf. Theorem \ref{thm:M b-nef}), and $K_X+\Delta_X+\frac{R_{\Delta_X}}{q-1}\sim_{\Q}0$ by \cite[Theorem 1.2]{G13}.

\par \vskip 1pc
\begin{proof1}[Proof of Claim \ref{claim 6.6}]
%\[K_X+\Delta_X+\frac{R_{\Delta_X}}{q-1}\sim_{\Q}\pi^*(K_Y+B+M).\]
We use the $(**)$ above.
For any prime divisor $P$ on $Y$, suppose $\pi^*P=m_P P'$, let
\[t_P=\sup \{t \, | \, (X,\Delta_X+\frac{R_{\Delta_X}}{q-1}+t\pi^*P) \, \text{is lc over generic point of }P\}.\]
Then
\[\mu_{P'}\Delta_X+\frac{\mu_{P'}R_{\Delta_X}}{q-1}+t_Pm_P\leq 1,\]
so
\[t_P\leq\frac{1}{m_P}(1-\frac{\mu_{P'}R_{\Delta_X}}{q-1}-\mu_{P'}\Delta_X).\]
Thus
\[\mu_PB=1-t_P\geq 1-\frac{1}{m_P}(1-\frac{\mu_{P'}R_{\Delta_X}}{q-1}-\mu_{P'}\Delta_X).\]
By Proposition \ref{prop:mori-fiber-gpair}, $\mu_{P'}(R_{\Delta_X}-\pi^*R_{\Delta_Y})=0$. In particular,
\[\mu_{P'}R_{\Delta_X}=m_P\mu_PR_{\Delta_Y}.\]
So 
\[\mu_P(\Delta_Y+\frac{R_{\Delta_Y}}{q-1})=\mu_P\Delta_Y+\frac{\mu_PR_{\Delta_Y}}{q-1}=\frac{m_P-1+\mu_{P'}\Delta_X}{m_P}+\frac{\mu_{P'}R_{\Delta_X}}{m_P(q-1)}\leq \mu_PB.\]
We obtain $B\geq \Delta_Y+\frac{R_{\Delta_Y}}{q-1}$.
This proves the claim and also the proposition.
\end{proof1}
\end{proof}

\begin{remark}
    We see from the proof that under the assumption in Proposition \ref{prop:mori-fiber-lcy}, $K_X+\Delta_X+\frac{R_{\Delta_X}}{q-1}\sim_{\Q}\pi^*(K_Y+\Delta_Y+\frac{R_{\Delta_Y}}{q-1}+M)$ is the canonical bundle formula of $(X,\Delta_X+\frac{R_{\Delta_X}}{q-1})$ with respect to $\pi$, where the moduli part $M\sim_{\Q}0$ by \cite[Theorem 1.3]{F14}.
\end{remark}

%\begin{proposition} Assuming \ref{conj Delta} holds when $\rho(X)=1$, then \ref{conj Delta} holds for projective klt $f$-pair. \end{proposition}

\begin{proof1} [Proof of Theorem \ref{thm:klt to Picard 1}]
We prove that the pair in (2) is log Calabi-Yau, by induction on $n =\dim X$. When $n = 1$,  Conjecture \ref{conj Delta} follows from Theorem \ref{thm:Delta+R/q-1 leq 1}.
%The case $n=1$ follows from Theorem \ref{Delta leq 1} (2).

Suppose Conjecture \ref{conj Delta} holds for klt pairs with dimension $<n$, and $\dim X=n$. We may assume $X$ is $\Q$-factorial by Lemma \ref{lemma:quasi etale or birational lcy} and Theorem \ref{thm:Q-fac-1}. If $K_X+\Delta$ is pseudo-effective, $K_X+\Delta\sim_{\Q} 0 $ by Proposition \ref{prop:T_f} (4), which implies  $R_\Delta\sim_{\Q}0$ (thus $=0$). It follows that $(X,\Delta+\frac{R_{\Delta}}{q-1})=(X,\Delta)$ is log Calabi-Yau by Corollary \ref{cor:quasi-'etale lcy}. 

Suppose $K_X+\Delta$ is not pseudo-effective. By \cite[Corollary 1.3.3] {BCHM} and Theorem \ref{thm:polarized-emmp}, after iteration of $f$, an $f$-equivariant $(K_X+\Delta)$-MMP reaches a Mori fiber space:
\[X=X_0\dashrightarrow X_1\dashrightarrow\cdots\dashrightarrow X_k\to X_{k+1} = Y\]
where $\pi_i$ on $(X_{i-1},\Delta_{i-1})$ (still klt) is birational for $1\leq i\leq k$, $\pi_{k+1}:(X_{k},\Delta_{k})\to Y$ is a Fano contraction. We prove that Conjecture \ref{conj Delta} holds for $(X_i,\Delta_i)$ by descending induction on $i$. Since $(X, \Delta_X)$ is an $f$-pair, all $(X_j, \Delta_j)$ are $f$-pairs by Proposition \ref{prop:birational-ramification-formula} and noting the case of flip is clear since it is isomorphism in Codimension $1$.

For $i=k$, if $\dim Y=0$, then $\rho(X_k)=1$, the conclusion follows from the assumption. If $\dim Y>0$, then $(F,\Delta_k|_F)$ is klt for a general fiber, and we can define $(Y,\Delta_Y)$ as in Definition \ref{def:mori-fiber-canonical} which is klt and an $f|_Y$-pair by Propositions \ref{prop:mori fiber klt} and \ref{prop:mori-fiber-gpair}.  Then our conclusion follows from Proposition \ref{prop:mori-fiber-lcy} and the case of dimension less than $n$.

Suppose Conjecture \ref{conj Delta} holds for $(X_i,\Delta_i)$. 
If $\pi_i:X_{i-1}\to X_i$ is a divisorial contraction, then Conjecture \ref{conj Delta} also holds for $(X_{i-1},\Delta_{i-1})$ by Proposition \ref{prop:birational-ramification-formula}. 
If $(\pi_i:X_{i-1}\dashrightarrow X_i$ is a flip, then Conjecture \ref{conj Delta} also holds for $(X_{i-1},\Delta_{i-1})$ by Lemma \ref{lemma:quasi etale or birational lcy} (3).

Hence the descending induction works and Conjecture \ref{conj Delta} holds for $(X,\Delta)$.
%$(X_{i-1},\Delta_{i-1}+\frac{R_{\Delta_{i-1}}}{q-1})$ is also log Calabi-Yau by Lemma
% \ref{lemma:quasi etale or birational lcy}(3).
 %let $\varphi_i:X_{i-1}\to Y_i$ be the flipping contraction, $\varphi_i^+:X_i\to Y_i$ the flipped contraction, and
%\[D_i=\varphi_{i*}(\Delta_{i-1}+\frac{R_{\Delta_{i-1}}}{q-1})=\varphi^+_{i*}(\Delta_{i}+\frac{R_{\Delta_{i}}}{q-1}).\]
%Since $K_X+\Delta_{i}+\frac{R_{\Delta_{i}}}{q-1}\sim_{\Q}0$,  it follows that
%\[K_{X_i}+\Delta_{i}+\frac{R_{\Delta_{i}}}{q-1}=\varphi^{+*}_i(K_{Y_i}+D_i),\]
%and
%\[K_{X_{i-1}}+\Delta_{i-1}+\frac{R_{\Delta_{i-1}}}{q-1}=\varphi^{*}_i(K_{Y_i}+D_i).\]
%Thus $(X_{i-1},\Delta_{i-1}+\frac{R_{\Delta_{i-1}}}{q-1})$ is log Calabi-Yau.
\end{proof1}

\begin{proof1}[Proof of Corollary \ref{cor:numerically trivial induction}]
For (1), suppose Conjecture \ref{conj Delta} holds for klt pairs with dimension $\leq n-2$. 
    %We can assume Conjecture \ref{conj Delta} (1) holds when $\dim X\leq n-1$ by induction.    
    As in the proof of Theorem \ref{thm:klt to Picard 1}, we may assume $X$ is $\Q$-factorial, $K_X+\Delta$ is not pseudo-effective, and we have an $f$-equivariant $(K_X+\Delta)$-MMP: $X = X_0 \dashrightarrow X_1 \dashrightarrow \cdots \dashrightarrow X_{k+1} =Y$.
    
    If $\dim Y\leq1$, then $f^*=q\id$ on $\N^1 (X)$ as in proof of \cite[Theorem 1.8 (4)]{MZ18}. Hence $f^*(K_X+\Delta)\equiv q(K_X+\Delta)$. 
    If $\dim Y\geq 2$, then $\dim X_k-\dim Y\leq n-2$. Define the $f|_Y$-pair $(Y,\Delta_Y)$ as in Definition \ref{def:mori-fiber-canonical}. By assumption, Conjecture \ref{conj Delta} holds for $(F,\Delta_k|_F)$ on a general $f$-periodic fiber $F$ of $X_k \to Y$. Also, %since $\dim Y<\dim X=n$, 
    we may assume that Conjecture \ref{conj Delta} (1) holds for $(Y,\Delta_Y)$ by induction like that for $(X, \Delta)$ to reduce to Conjecture \ref{conj Delta} ($\dim \le n-2$).
    
   % We run log MMP for $Y$, do induction on $\dim Y$ and also the descending birational contractions as in the proof of Theorem \ref{thm:klt to Picard 1}, hence we may assume Conjecture \ref{conj Delta} (1) holds for $Y$.
%and Conjecture \ref{conj Delta}(1) holds for $(Y,\Delta_Y)$,
As in the proof of Proposition \ref{prop:mori-fiber-lcy},
we have
    \[0 \equiv K_{X_k}+\Delta_{k}+\frac{R_{\Delta_{k}}}{q-1}\sim_{\Q}\pi^*(K_Y+\Delta_Y+\frac{R_{\Delta_Y}}{q-1}+M)\]
    with $M\equiv 0$.
    %as in proof of Proposition \ref{prop:mori-fiber-lcy}. Hence $K_{X_k}+\Delta_{k}+\frac{R_{\Delta_{k}}}{q-1}\equiv0$. 
    Then the corollary follows from the same descending induction on $X_i$ ($1\leq i\leq k$) as in the proof of Theorem \ref{thm:klt to Picard 1}.
    
    For (2), by Lemma \ref{lemma:quasi etale or birational lcy} (4) and Theorem \ref{thm:Q-fac-1}, we may assume that $X$ is $\Q$-factorial and the $f$-pair $(X, \Delta)$ is lc
    (cf. Theorem \ref{thm:(X,Delta) lc}). Then the same induction and log MMP $X = X_0 \dashrightarrow X_1 \dashrightarrow \cdots \dashrightarrow X_{k+1} =Y$
    as in proof of (1) above or Theorem \ref{thm:klt to Picard 1}, works.
   Note that when $\dim Y\geq2$, the general fiber $F$ of $X_k \to Y$ is of dimension $\le 1$ (hence, being normal, is smooth), so Conjecture \ref{conj Delta} holds for $(F,\Delta_k|_F)$.
\end{proof1}

\section{Application to surfaces; Proof of Theorem \ref{thm: conj delta surface}}\label{section: surface appl}

In this section, we prove Theorem \ref{thm: conj delta surface}. We need two results.

\begin{lemma}\label{lemma:disturb Delta}
    Let $f: X \to X$ be a surjective endomorphism of a normal projective variety and $\Delta$ an effective Weil $\Q$-divisor such that $(X,\Delta)$ is an $f$-pair. Suppose $\Delta=\sum_{i=1}^{k}a_iP_i$ with $P_i$ being distinct prime divisors, and $f^{*}(P_j)=qP_j$ for $1\leq j \leq s$, $s\leq k$, and some $q \ge 2$. Let $\Delta'=\sum_{i=1}^{s}b_iP_i+\sum_{i=s+1}^{k}a_iP_i$. Then $(X,\Delta')$ is an $f$-pair whenever $0 \le b_i \le 1$. Further, 
%0\leq b_1,...,b_s\leq 1$, 
$\Delta'+\frac{R_{\Delta'}}{q-1}=\Delta+\frac{R_{\Delta}}{q-1}$.
\end{lemma}
\begin{proof}
    This is because $R_{\Delta}=\Delta+R_f-f^*\Delta$, and note that $\mu_{P_j}R_{\Delta'} = (q-1) (1 - b_j)$.
\end{proof}
 
\begin{theorem}[{cf.~\cite[Theorems 2.8 and 2.9]{W90}}]\label{thm:(X,Delta) lc surface}
    Let $f: X \to X$ be a non-isomorphic surjective endomorphism of a normal projective variety and $\Delta$ an effective Weil $\Q$-divisor such that $(X,\Delta)$ is an $f$-pair. Then $K_X+\Delta$ is $\Q$-Cartier and $(X,\Delta)$ is lc.
\end{theorem}

\begin{proof}
    By \cite[Proposition 3.5]{F12}, it suffices to show that $(X,\Delta)$ is numerically lc. Let $\Nlc_{num}(X,\Delta)$ be the numerically non-lc locus of $(X,\Delta)$ (see \cite[section 20]{CMZ20}). By the proof of \cite[Lemma 10.1]{CMZ20} or Remark \ref{rmk: proof of num lc} (2), $\Nlc_{num}(X,\Delta)\bigcap \Supp R_{\Delta}=\varnothing$. Now the theorem follows from the proof of Wahl \cite[Theorems 2.8]{W90} with $(X,0)$ replaced by $(X,\Delta)$.
\end{proof} 

\begin{proof1}[Proof of Theorem \ref{thm: conj delta surface}]
    By Theorem \ref{thm:(X,Delta) lc surface}, the pair $(X, \Delta)$ is lc.
   % $K_X+\Delta$ is $\Q$-Cartier  
   If $K_X+\Delta$ is pseudo-effective, then $K_X+\Delta\sim_{\Q} 0$ by Proposition \ref{prop:T_f} (4), so the pair $(X,\Delta+\frac{R_{\Delta}}{q-1})=(X,\Delta)$ is log Calabi-Yau.
    Suppose $K_X+\Delta$ is not pseudo-effective. After iteration of $f$, we may assume $f^*=q\, \id^*$ on $\N^1(X)$ by \cite[Theorem 2.7]{Z10}. So $f^*(K_X+\Delta)\equiv q(K_X+\Delta)$, i.e., $K_X+\Delta+\frac{R_{\Delta}}{q-1}\equiv0$.
   Let $\pi:X\to Y$ be a step of $(K_X+\Delta)$-MMP, which we may assume to be $f$-equivariant, after iteration of $f$.

    \textit{Case (1):} $\dim Y=1$. Then $\pi$ is a Fano contraction. Define $(Y,\Delta_Y)$ as in Definition \ref{def:mori-fiber-canonical}.  
    Since Conjecture \ref{conj Delta} holds for curve pairs $(Y,\Delta_Y)$ and $(F,\Delta_F)$ with $F$ a general fiber of $\pi$ by Theorem \ref{thm:Delta+R/q-1 leq 1}, Conjecture \ref{conj Delta} holds for $(X,\Delta)$ by Proposition \ref{prop:mori-fiber-lcy}.

    \textit{Case (2):} $\dim Y=0$ or $2$. If $\dim Y=0$, then $\rho(X)=1$, so $T_f\neq 0$ from the assumption. If $\dim Y=2$, then $\pi:X\to Y$ contracts a curve, which is $f^{-1}$-periodic. So we may assume $T_f\neq 0$ in Case (2). For each prime divisor $P\subseteq \Supp T_f$, we may further assume $f^{-1}(P)=P$ after iteration of $f$, and  $\mu_P\Delta=1$ by Lemma \ref{lemma:disturb Delta}. Note that $(X,\Delta)$ remains lc by Theorem \ref{thm:(X,Delta) lc surface}. Since $P$ is an lc centre of the pair $(X,\Delta)$, applying Theorem \ref{thm:polarized-inversion-of-adjunction} and noting that Conjecture \ref{conj Delta} holds for curve pairs,
    %to $\widetilde{P}$ (the normalization of $P$), 
    $(X,\Delta + \frac{R_{\Delta}}{q-1})$ is lc near $P$. Thus, by
    the normality of $X$ and 
    Theorem \ref{Delta leq 1}, $\Nlc := \Nlc(X,\Delta+\frac{R_{\Delta}}{q-1})$ consists of several points
    $Q_j$ away from $\Supp T_f$, and $\Nklt: = \Nklt(X,\Delta+\frac{R_{\Delta}}{q-1})$ consists of $\Supp T_f$, $Q_j$ and several other points. If such $Q_j$ exists, then $\Nklt$ is not connected
    which contradicts \cite[Theorem 1.2 (2)]{B24}.
    Thus $\Nlc = \emptyset$. This proves the theorem.
    %$\Supp T_f$ is the union of dimension one components of $\Nklt(X,\Delta+\frac{R_{\Delta}}{q-1})$. If $\Nlc(X,\Delta+\frac{R_{\Delta}}{q-1})$ has an isolated point $Q$ then it is away from $T_f$, so $\Nklt(X,\Delta+\frac{R_{\Delta}}{q-1})$. So $\Nlc(X,\Delta+\frac{R_{\Delta}}{q-1})\bigcap \Supp T_f=\varnothing$. Suppose $\Nlc(X,\Delta+\frac{R_{\Delta}}{q-1})\neq \varnothing$. Then $\dim\Nlc(X,\Delta+\frac{R_{\Delta}}{q-1})=0$ by Theorem \ref{thm:Delta+R/q-1 leq 1}. Thus $\Nklt(X,\Delta+\frac{R_{\Delta}}{q-1})$ is not connect, which contradicts \cite[Theorem 1.2]{B24}.
\end{proof1}
\appendix
\renewcommand\thethm{A\arabic{section}.\arabic{thm}}
\setcounter{thm}{0}

\section{Equivariant $\Q$-factorial models} \label{appen:equi-Q-fac-model}

Let $X$ be a normal variety. We say $\pi:\widetilde{X}\rightarrow X$ is a small $\Q$-factorial model of $X$, if $\widetilde{X}$ is $\Q$-factorial, $\pi$ is projective and isomorphic in codimension 1. The following result of Moraga-Y\'a\~nez-Yeong \cite{MYY24}
says that an endomorphism can be lifted to some small $\Q$-factorial model when $X$ is of klt type.

\begin{thm}[{cf. \cite[Theorem 1.2]{MYY24}}]\label{thm:Q-fac-1}
Let $(X,\Delta)$ be an klt projective pair, and $f:X\to X$ a surjective endomorphism. Then there exists a small $\Q$-factorial model $\pi: \widetilde{X}\rightarrow X$, and an endomorphism $\widetilde{f}$ on $\widetilde{X}$ such that $\pi\circ \widetilde{f}=f^s\circ \pi$ for some $s>0$.
\end{thm}

\begin{proof}
We may assume $X$ is not $\Q$-factorial, otherwise, there is nothing to prove. Then we define a sequence of $Y_i$ inductively in the following way.

\begin{tikzcd}
Y_i \arrow[r, ]  &
Z_{i} \arrow[d, ] \arrow[r, ] & 
Y_{i-1} \arrow[d, "\pi_{i-1}"] \\
                &
X \arrow[r, "f"] & 
X
\end{tikzcd}
Let $Y_0$ be a small $\Q$-factorial model of $X$ \cite[Corollary 1.4.3]{BCHM}. Suppose $\pi_{i-1}:Y_{i-1}\rightarrow X$ is already defined, which is a small $\Q$-factorial model. Let $Z_{i}$ be the normalization of $Y_{i-1}\times_X X$, with respect to $\pi_{i-1}:Y_{i-1}\rightarrow X$ and $f: X \rightarrow X$. Then $Z_{i} \rightarrow X$ (the domain of $f$) is isomorphic in codimension one. Thus $Z_{i}$ is of klt type. Choose a small $\Q$-factorial model $Y_{i}\rightarrow Z_{i}$, which is also a small $\Q$-factorial model of $X$. Let $f_i$ denote the composition $Y_i \to Z_i \to Y_{i-1}$,  and $\pi_i$ denote $Y_i\to X$ (the domain of $f$).

By \cite[Lemma 2.8]{FHS25}, small $\Q$-factorial model of $X$ is finite up to isomorphism over $X$, so we must have $\phi:Y_i\cong Y_j$ over $X$ for some $i<j$. Then we can take $\widetilde{X}=Y_i$, $\widetilde{f}=f_{i+1}\circ \cdots\circ f_{j-1} \circ f_{j}\circ \phi$, and $s=j-i$.
\end{proof}

In \cite[Theorem 1.2]{MYY24}, the $\Q$-factorial model depends on $f$. Actually, $f$ can be lifted to any small $\Q$-factorial model. For this, we need the following, which says the ``anti-Stein" factorization of a generically finite morphism into ``finite'' followed by ``birational'' is unique up to isomorphism.

\begin{lemma}\label{lemma:gen-finite-morphism}
    Let $Y\to X$ be a generically finite surjective morphism between normal projective varieties. Suppose that $g_i:Y\to Z_i$, and $h_i:Z_i\to X$, $i=1,2$, satisfy $f=h_i\circ g_i$, where the $g_i$ are finite surjective and $h_i$ are projective birational. Then an isomorphism $\varphi:Z_1\to Z_2$ exists, such that $g_2=\varphi\circ g_1$ and $h_1=h_2\circ \varphi$.
\end{lemma}

\begin{proof}
\mbox{}
\begin{tikzcd}
W\arrow[ d, "p_Z" ]  \arrow[r, "\eta"]  \arrow[ rr, bend left ,  "p_S"]&
Y\arrow[ d, "g_i" ]  \arrow[r, "\alpha" ]     &
S\arrow[ d, "\beta" ]    \\
Z\arrow[ r, "\tau_i"]        &
Z_i\arrow[r, "h_i"]      &
X
\end{tikzcd}
Let $\alpha: Y\to S$, $\beta: S\to X$ be the Stein factorization of $Y\to X$. Then $(*)$:  $Y$ is isomorphic to the normalization of $Z_i$ in $k(Y) = k(S)$.
    %S\times_X Z_i$ ($i=1,2$). 
    Let $Z$ be a projective birational model which dominates $Z_i$ via $\tau_i:Z\to Z_i$, 
    %$W_i$ the normalization of $Y\times_{Z_i}Z$
    %(also the normalization of $S$ in the function
    %field $k(Y)$), 
    and 
    $W$ the normalization of $S\times _X Z$ (also the normalization of $Z$ in the function field $k(S) = k(Y)$). Let $p_S: W \to S$ and $p_Z: W \to Z$ be the projections.
    By the $(*)$ above, there is a birational morphism $\eta: W \to Y$ such that the composition $\alpha \circ \eta: W \to Y \to S$ equals $p_S$, and $(**)$: the composition $g_i \circ \eta: W \to Y \to Z_i$ is just $\tau_i \circ p_Z$. 
    %and both the compositions $h_i \circ \tau_i \circ p_Z: W \to Z \to Z_i \to X$ equal the composition $\beta \circ p_S: W \to S \to X$.
    %then $W_1\cong W\cong W_2$. 
    Since $g_i: Y\to Z_i$ and $p_Z: W\to Z$ are finite, a curve $p_Z(C)$ is contracted by $\tau_i$ if and only if 
    $C$ is contracted by $\eta$.
    Thus both $\tau_i: Z\to Z_i$ ($i=1,2$) contract the same curves, which implies an isomorphism $\varphi:Z_1\cong Z_2$ by the rigidity lemma (cf.~\cite[Lemma 1.15]{D01}),
    such that $\varphi\circ \tau_1=\tau_2$. By the $(**)$, $\varphi$ satisfies the condition.
\end{proof}

\begin{thm}\label{thm:Q-fac-2}
Let $(X,\Delta)$ be an klt projective pair, and $f:X\to X$ a surjective endomorphism. Then for any  small $\Q$-factorial model $\pi: \widetilde{X}\rightarrow X$, an endomorphism $\widetilde{f}$ on $\widetilde{X}$ exists such that $\pi\circ \widetilde{f}=f^s\circ \pi$ for some $s>0$.
\end{thm}

\begin{proof}
    Assume $X$ is not $\Q$-factorial. As in Theorem \ref{thm:Q-fac-1}, there is a sequence of commutative diagrams:
%\begin{center} \begin{tikzcd} Y_{i} \arrow[d, "\pi_i" ] \arrow[r, "f_i" ] & Y_{i-1} \arrow[d, "\pi_{i-1}"] \\ X \arrow[r, "f"] & X \end{tikzcd} \end{center}
 $\pi_{i-1} \circ f_i : Y_i \to Y_{i-1} \to X$ equals $f \circ \pi_i: Y_i \to X \to X$,
and $Y_i\cong Y_{i+s}$ over $X$ for some $i, s \ge 1$. For any $k$, by taking a common resolution $Y$ of $Y_k$ and $Y_{0}$, we see that $\rho(Y_k)=\rho(Y_{0})=\rho(Y)-m$, with $m$ the number of exceptional prime divisors of $Y$ over $X$. Then the composition $f_{1}\circ \cdots \circ f_{k}$ (thus each $f_j$) is finite.
If we identify $Y_{k+1} = Y_{s+k+1}$ over $X$, 
%$k\in \Z_{\geq 0}$, 
we have two ``anti-Stein" factorizations $\pi_{s+k} \circ f_{s+k+1}: Y_{s+k+1} \to Y_{s+k}\to X$ and $\pi_{k} \circ f_{k+1}: Y_{k+1} \to Y_k\to X$. The uniqueness in Lemma \ref{lemma:gen-finite-morphism} implies $Y_k\cong Y_{s+k}$ over $X$ and inductively, $Y_0\cong Y_s$ over $X$. Thus $f^s$ lifts to $Y_0$.
\end{proof}


\begin{thebibliography}{99}
\bibitem{A99}
F.~Ambro, The adjunction conjecture and its applications,~ProQuest~LLC,~Ann Arbor, MI,~1999.

%\bibitem{A04}
%F. Ambro, Shokurov's boundary property, J. Differential Geom. {\bf 67} (2004), no.~2, 229--255.

\bibitem{BFMT25}
B. Bakker, S. Filipazzi, M. Mauri, J. Tsimerman, Baily--Borel compactifications of period images and the b-semiampleness conjecture, arXiv:2508.19215

\bibitem{B19}
C. Birkar, Anti-pluricanonical systems on Fano varieties, Ann. of Math. (2) {\bf 190} (2019), no.~2, 345--463.

\bibitem{B24}
C. Birkar, On connectedness of non-klt loci of singularities of pairs, J. Differential Geom. {\bf 126} (2024), no.~2, 431--474.

\bibitem{BCHM}
C. Birkar, P. Cascini, C. Hacon, J. McKernan, Existence of minimal models for varieties of log general type, J. Amer. Math. Soc. {\bf 23} (2010), no.~2, 405--468. 

\bibitem{BZ16}
C. Birkar and D.-Q. Zhang, Effectivity of Iitaka fibrations and pluricanonical systems of polarized pairs, Publ. Math. Inst. Hautes \'Etudes Sci. {\bf 123} (2016), 283--331.

\bibitem{BG17}   
A. Broustet and Y. Gongyo,
Remarks on log Calabi-Yau structure of varieties admitting polarized endomorphisms, 
Taiwanese J. Math. {\bf 21} (2017), no.~3, 569--582.

\bibitem{BH14}
A. Broustet and A. H\"oring, Singularities of varieties admitting an endomorphism, Math. Ann. {\bf 360} (2014), no.~1-2, 439--456.

\bibitem{CMZ20}
P. Cascini, S. Meng and D.-Q. Zhang, Polarized endomorphisms of normal projective threefolds in arbitrary characteristic, Math. Ann. {\bf 378} (2020), no.~1-2, 637--665.

\bibitem{D01}
O. Debarre, {\it Higher-dimensional algebraic geometry}, Universitext, Springer, New York, 2001.

\bibitem{F03}
N. Fakhruddin, Questions on self maps of algebraic varieties, J. Ramanujan Math. Soc. {\bf 18} (2003), no.~2, 109--122.

\bibitem{FHS25}
S. Filipazzi, C.~D. Hacon and R. Svaldi, Boundedness of elliptic Calabi--Yau threefolds, J. Eur. Math. Soc. (JEMS) {\bf 27} (2025), no.~9, 3583--3650.

\bibitem{F14}
E. Floris, Inductive approach to effective b-semiampleness, Int. Math. Res. Not. IMRN {\bf 2014}, no.~6, 1465--1492.

\bibitem{F12}
O. Fujino, Minimal model theory for log surfaces, Publ. Res. Inst. Math. Sci. {\bf 48} (2012), no.~2, 339--371.

\bibitem{FH22}
O. Fujino and K. Hashizume, On inversion of adjunction, Proc. Japan Acad. Ser. A Math. Sci. {\bf 98} (2022), no.~2, 13--18.

\bibitem{FH23}
O. Fujino and K. Hashizume, Adjunction and inversion of adjunction, Nagoya Math. J. {\bf 249} (2023), 119--147.

\bibitem{G13}
Y. Gongyo, Abundance theorem for numerically trivial log canonical divisors of semi-log canonical pairs, J. Algebraic Geom. {\bf 22} (2013), no.~3, 549--564.

%\bibitem{GOST15}
%Y. Gongyo, S. Okawa, A. Sannai, S. Takagi, Characterization of varieties of Fano type via singularities of Cox rings, J. Algebraic Geom. {\bf 24} (2015), no.~1, 159--182.

\bibitem{H77}
R. Hartshorne, {\it Algebraic geometry}, Graduate Texts in Mathematics, No. 52, Springer, New York-Heidelberg, 1977.

\bibitem{I82}
S. Iitaka, {\it Algebraic geometry}, North-Holland Mathematical Library Graduate Texts in Mathematics, 24 76, Springer, New York-Berlin, 1982.

\bibitem{K07}
M. Kawakita, Inversion of adjunction on log canonicity, Invent. Math. {\bf 167} (2007), %no.~1, 
129--133.

\bibitem{K92}
J. Koll\'ar, et al., Flips and abundance for algebraic threefolds, Soci\'et\'e{} Math\'ematique de France, Ast\'erisque,  Vol. 211, 1992.

\bibitem{KM98}
J. Koll\'ar and S. Mori, { Birational geometry of algebraic varieties}, translated from the 1998 Japanese original, 
Cambridge Tracts in Mathematics, 134, Cambridge Univ. Press, Cambridge, 1998.

\bibitem{M20}
S. Meng, Building blocks of amplified endomorphisms of normal projective varieties, Math. Z. {\bf 294} (2020), no.~3-4, 1727--1747.

\bibitem{M23}
S. Meng, Log Calabi-Yau structure of projective threefolds admitting polarized endomorphisms, Int. Math. Res. Not. IMRN {\bf 2023}, no.~24, 21272--21289, arXiv:2204.11244v3.

\bibitem{MZ18}
S. Meng and D.-Q. Zhang, Building blocks of polarized endomorphisms of normal projective varieties, Adv. Math. {\bf 325} (2018), 243--273.

\bibitem{MZ20}
S. Meng and D.-Q. Zhang, Semi-group structure of all endomorphisms of a projective variety admitting a polarized endomorphism, Math. Res. Lett. {\bf 27} (2020), no.~2, 523--549.

\bibitem{MZ22}
S. Meng and D.-Q. Zhang, Kawaguchi-Silverman conjecture for certain surjective endomorphisms, Doc. Math. {\bf 27} (2022), 1605--1642.

\bibitem{MiZ88}
M. Miyanishi and D.-Q. Zhang, Gorenstein log del Pezzo surfaces of rank one, J. Algebra {\bf 118} (1988), no.~1, 63--84.

\bibitem{MYY24}
J. Moraga, J.~I. Y\'a\~nez and W. Yeong, Polarized endomorphisms of log Calabi-Yau pairs, arXiv:2406.18092v2.

\bibitem{NZ10}
N. Nakayama and D.-Q. Zhang, Polarized endomorphisms of complex normal varieties, Math. Ann. {\bf 346} (2010), no.~4, 991--1018, arXiv:0908.1688v1.

\bibitem{OX12}
Y. Odaka and C. Xu, Log-canonical models of singular pairs and its applications, Math. Res. Lett. {\bf 19} (2012), no.~2, 325--334.

\bibitem{R83}
M. Reid, Projective morphisms according to Kawamata, preprint, Univ. of Warwick, 1983.

\bibitem{R94}
M. Reid, Nonnormal del Pezzo surfaces, Publ. Res. Inst. Math. Sci. {\bf 30} (1994), no.~5, 695--727.

\bibitem{W90}
J.~M. Wahl, A characteristic number for links of surface singularities, J. Amer. Math. Soc. {\bf 3} (1990), no.~3, 625--637.

\bibitem{Y21}
S. Yoshikawa, Structure of Fano fibrations of varieties admitting an int-amplified endomorphism, Adv. Math. {\bf 391} (2021), Paper No. 107964, 32 pp.

\bibitem{Z10}
D.-Q. Zhang, Polarized endomorphisms of uniruled varieties, Compos. Math. {\bf 146} (2010), no.~1, 145--168.

%\bibitem{Z14}
%D.-Q. Zhang, Invariant hypersurfaces of endomorphisms of projective varieties, Adv. Math. {\bf 252} (2014), 185--203.


\end{thebibliography}
\end{document}